\documentclass[10pt]{article}

\usepackage[utf8]{inputenc}

\usepackage{hyperref}
\usepackage{verbatim} %to allow long "comment-out-'s"

\usepackage{graphicx}
\usepackage{caption}
\usepackage{subcaption}
\usepackage{amssymb}
\usepackage{ textcomp } %Promille sign
\usepackage{color}
\usepackage{enumerate}
\usepackage{mathrsfs}
\usepackage{multicol}

\usepackage[ruled,vlined, onelanguage]{algorithm2e}

\usepackage{amsopn}

\usepackage{amsmath}
\usepackage{amsthm}
\usepackage{amssymb}

\usepackage{multicol}

\usepackage[ruled,vlined]{algorithm2e}

\newcommand{\R}{\mathbb{R}}

\newcommand{\calA}{\mathcal{A}}

\newcommand{\calC}{\mathcal{C}}
\newcommand{\calD}{\mathcal{D}}

\newcommand{\calL}{\mathcal{L}}
\newcommand{\calM}{\mathcal{M}}
\newcommand{\calP}{\mathcal{P}}

\newcommand{\abs}[1]{\vert #1 \vert}
\newcommand{\norm}[1]{\Vert #1 \Vert}

\newcommand{\set}[1]{\left\lbrace #1\right\rbrace}
\newcommand{\sse}{\subseteq}
\newcommand{\sprod}[1]{\left\langle #1 \right\rangle}

\newcommand{\sph}{\mathbb{S}}

\newcommand{\prb}[1]{\mathbb{P}\left( #1 \right)}
\newcommand{\erw}[1]{\mathbb{E}\left( #1 \right)}

\newcommand{\impl}{\Rightarrow}

\usepackage{dsfont}
\newcommand{\ind}{\mathds{1}}

\newcommand{\geqsim}{\gtrsim}
\newcommand{\leqsim}{\lesssim}

\DeclareMathOperator{\cone}{cone}

\DeclareMathOperator{\dist}{dist}
\DeclareMathOperator{\spn}{span}

\DeclareMathOperator{\supp}{supp}

\DeclareMathOperator{\argmin}{argmin}

\DeclareMathOperator{\ran}{ran}
\DeclareMathOperator{\pos}{pos}

\newtheorem{lem}{Lemma}
\newtheorem{prop}[lem]{Proposition}

\newtheorem{theo}[lem]{Theorem}

\newtheorem{rem}[lem]{Remark}

\numberwithin{lem}{section}

\title{Soft Recovery Through $\ell_{1,2}$ Minimization with Applications in Recovery of Simultaneously Sparse and Low-Rank Matrices\thanks{This work has been submitted to the IEEE for possible publication. Copyright may be transferred without notice,
after which this version may no longer be accessible.}}

\usepackage[affil-it]{authblk}

\author{Axel Flinth\thanks{E-mail: {\tt flinth@math.tu-berlin.de}}}
\affil{Institut für Mathematik \\ Technische Universität Berlin}

\begin{document}

\maketitle

\begin{abstract}
This article provides a new type of analysis of a compressed-sensing based technique for recovering column-sparse matrices, namely minimization of the $\ell_{1,2}$-norm. Rather than providing conditions on the measurement matrix which guarantees the solution of the program to be exactly equal to the ground truth signal (which already has been thoroughly investigated), it presents a condition which guarantees that the solution is approximately equal to the ground truth. \emph{Soft recovery statements} of this kind are to the best knowledge of the author a novelty in Compressed Sensing. Apart from the theoretical analysis, we present two heuristic proposes how this property of the $\ell_{1,2}$-program can be utilized to design algorithms for recovery of matrices which are sparse and have low rank at the same time.


\end{abstract}

\section{Introduction}

%% TO INCLUDE

During the course of the last decade, the concept of \emph{Compressed Sensing} \cite{CandesTao2005} has gained an enormous interest in the signal processing community. Put shortly, Compressed Sensing is the science on how one can solve underdetermined equations using structural assumptions on the solutions (ground truth signals). Originally, the problem of recovering sparse vectors $x_0 \in \R^n$ from linear measurements $Ax_0$ were considered, but the ideas have been used to extend the theories to many other settings. 
In this paper, we will stay in the linear measurement regime, but consider another type of signal, namely \emph{column-sparse matrices}. A matrix $X \in \R^{k,n}$ is thereby said to be \emph{$s$-column sparse} when only $s$ of its columns are non-zero.  This problem appears naturally in a certain instance of the so called \emph{blind deconvolution problem}. For motivational purposes, let us begin by describing this in a bit greater detail.

\subsection{Blind Deconvolution and Column-sparse Matrices}
Fundamentally, the problem of blind deconvolution reads as follows: From observing the convolution $v= w*x$ of two signals $w$ and $x$, reconstruct $w$ and $x$. Without any structurial assumptions on $w$ and $x$, we cannot expect to succeed -- this is already clear from considering the dimensions of the problem.

There is a standard way of transforming the blind deconvolution to a matrix recovery problem \cite{ahmed2014blind}. It goes as follows: First, take the Fourier transform of the equation $v= w * x$:
\begin{align*}
\hat{v} = \hat{w} \odot \hat{x},
\end{align*}
where $\odot$ refers to pointwise multiplication, i.e. $(\hat{w}_i \odot \hat{x})_i = \hat{w}_i \hat{x}_i$. Now we make two assumptions: The vector $\hat{w}$ lies in a $k$-dimensional subspace of $\R^n$ with basis $(d_i)_i$, and $\hat{x}$ is sparse in some basis $(e_i)_i$ of $\R^n$.  Both the subspace and the basis are known. In applications, this could for instance mean that $w$ is known to be bandlimited, and that $x$ is a result of a user transmitting a short linear combination of some pre-determined main modes. Setting $\hat{w} = \sum_j h_j d_j$, and $\hat{x}= \sum_\ell z_\ell e_\ell$ and expanding our equation $\hat{v}_i = \hat{w}_i \hat{x}_i$, we obtain:
\begin{align*}
	\hat{v}_i = \bigg(\sum_{j=1}^k h_j d_j \bigg)_i \bigg(\sum_{\ell=1}^n z_\ell e_\ell \bigg)_i = \sum_{j,\ell} d_j(i)e_\ell(i) h_jz_\ell = \calA(h z^*), \quad i=1, \dots n
\end{align*}
where we defined the linear map $\calA$ through its action on an $(k \times n)$-matrix
\begin{align*}
	\calA(M) = \sum_{j,\ell} d_j(i)e_\ell(i)M_{j,\ell}.
\end{align*} 
Hence, after proper reformulation, the \emph{non-linear} blind deconvolution problem of recovering a \emph{pair of vectors} can be seen as the problem of recovering a \emph{matrix} from \emph{linear} measurements $y = \calA(M)$.
It is also evident that matrices we want to recover, i.e. matrices of the form $Z_0=h_0z_0^*$ with $z_0 \in \R^n$ sparse and $h_0 \in \R^k$, are column sparse. One should also note that $h_0z_0^*$ is low-rank - in fact, all columns are scaled versions of the vector $h_0$. Although we will primarily focus on the column-sparse structure in this work, we will also discuss the low rank aspect of the problem.

One could also imagine that one observes the sum of several convolutions with unknown filters: $v = \sum_{\ell=1}^r w_\ell * x_\ell$. One then sometimes speak of a \emph{blind deconvolution and demixing} problem. A similar transformation of this problem leads to the task of recovering a matrix $Z_0 = \sum_{\ell=1}^r h_\ell^0 (z^0_\ell)^*$ from linear measurements. Assuming that the vectors $z^0_\ell$, $\ell=1, \dots r$ are sparse, $Z_0$ again becomes column sparse. Also, it trivially has rank at most $r$. 

\subsection{$\ell_{1,2}$-Minimization and Soft Recovery}

Now let us return to the general problem of recovering a column-sparse matrix $Z_0 \in \R^{k,n}$ from linear measurements $y=\calA(Z_0)$. A canonical approach \cite{Eldar2009BlockSparse} to solving this problem is to use $\ell_{1,2}$-minimization, i.e. to solve the problem
\begin{align*}
	\min \norm{Z}_{1,2} \text{ subject to } \calA(Z)= y. \tag{$\calP_{1,2}$}
\end{align*}
The $\ell_{1,2}$-norm of a matrix is thereby defined as the sum of the Euclidean norm of its columns. There is much known about this problem -- see for instance \cite{Eldar2009BlockSparse, stojnic2008reconstruction}. 

Performing small numerical experiments with $\calA: \R^{k,n} \to \R^m$ Gaussian, one can observe that $\calP_{1,2}$ recovers the ground truth signals approximately already for values for $m$ for which the solution $\widehat{Z}$ of $\calP_{1,2}$ is almost never equal to $Z_0$. In particular, one notices that the solution $\widehat{Z}$ of of $\calP_{1,2}$ seems to have the following behaviour:

\begin{itemize}
	\item The directions of the columns $\widehat{Z}(i)$ of $\widehat{Z}(i)$ are well-aligned with the directions of the columns of $Z_0(i)$ (for the indices corresponding to non-zero columns of $Z_0$).
	
	\item The energy of $\widehat{Z}$ is concentrated on the columns corresponding to the non-zero columns of $Z_0$. That is: If $Z_0(i)=0$, $\norm{\widehat{Z}(i)}$ will be small, and vice versa.
	
	\item Also, if $Z_0$ has rank $r$, the subspace spanned by the first $r$ left singular vectors of $\widehat{Z}$ is close to $\ran Z_0$.
	
	\end{itemize}
	See also Figure \ref{fig:intuition}. We will say, in a consciously relatively unprecise manner, that $Z_0$ has been \emph{softly recovered} by the program $\calP_{1,2}$ if the solution $\widehat{Z}$ behaves as described above. The aim of this work is to provide a theoretical explanation of this feature of $\calP_{1,2}$. In particular, we want to state conditions on $\calA$ which imply soft recovery, but are weaker than the ones for exact recovery.
	
To the best knowledge of the author, this is the first time an analysis of this type has been conducted for any compressed sensing problem. One major reason for this is probably that the  ''original'' Compressed Sensing problem of $\ell_1$-minimization;
\begin{align*}
	\min \norm{z}_1 \text{ subject to } Az=y, \tag{$\calP_1$}
\end{align*} 
\emph{does not} exhibit soft recovery. In fact, Proposition \ref{prop:l1supp} in the Appendix of this paper shows for each matrix $A$ and each ground truth signal $z_0$, either the solution of the problem $\calP_1$ is equal to $z_0$, or there exists as at least one index $i$ in the support of $z_0$ such that $\hat{z}(i) =0$. \\

	The reasons for conducting this analysis is twofold. First and foremost, we want to fundamentally understand the reason behind the soft recovery phenomenon. The technique used for the proof is also quite general, and could possibly be applied to other optimization-based recovery programs than $\calP_{1,2}$ in the future.
	
	The second reason is that if soft recovery is present, the solution $\widehat{Z}$ of $\calP_{1,2}$ will give us information about the column-support and range of the ground truth solution $Z_0$ also when using too few measurements to recover it exactly. We will in fact use this fact to design two heuristic algorithms for recovery of column-sparse, low-rank matrices, and test their performance numerically. \\
	
	The reader should not confuse soft recovery statements with statements about \emph{instance optimality} \cite{wojtaszczyk2010stability}. The latter are statements of the form ''if the ground truth signal is \emph{approximately} column-sparse, $\calP_{1,2}$ will \emph{approximately} recover it''. In contrast, we will consider \emph{exactly} column-sparse ground truth signals, and investigate when they are \emph{approximately} recovered.

\begin{figure}
	\begin{center}
		\includegraphics[scale=.5]{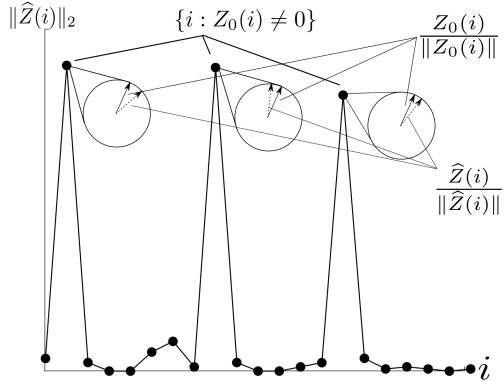}
		\end{center}
		\caption{The aim of this work is to prove that the situation this figure depicts is the typical one. The vectors in the small circles show the directions of the columns of $Z_0(i)$ and $\widehat{Z}(i)$, respectively, for indices $i$ with $Z_0(i)\neq 0$. Note that $k=2$ in the figure. \label{fig:intuition}}
	\end{figure}

\subsection{Related Work on Low-Rank, Column-Sparse Matrix Recovery.} 

One should note that there are more sophisticated approaches to recovering low-rank and column-sparse matrices than $\ell_{1,2}$-minimization. For the sake of completeness, let us briefly describe two of these. 

When recovering $Z_0$, we should try to take both the column sparsity and the low rank assumption into account. Since it is well known that low rank is promoted by the nuclear norm \cite{recht2010guaranteed}, defined as the sum of the singular values of the matrix,  a natural approach would be to minimize a weighted sum of the two norms
\begin{align*}
	\min \norm{Z}_{1,2} + \lambda \norm{Z}_{*} \text{ subject to } \mathcal{A}(Z)=y.
\end{align*}
This approach does perform reasonably well, but it has its flaws. From a practical point of view, choosing the correct $\lambda$ is a non-trivial task. Also, the nuclear-norm minimization procedure tend to be a lot slower than procedures for minimizing norms like $\ell_1$ or $\ell_{1,2}$. From a more theoretical standpoint, any such mixed-norm approach will need as least the same order of measurements to recover a low-rank, sparse matrix $Z_0$, as the minimum of the amounts needed to recover $Z_0$ only with help of $\norm{\cdot}_{1,2}$ and $\norm{\cdot}_{*}$, respectively. This 'single structure bottleneck' holds in much more generality - see \cite{oymak2015simultaneously}. The latter fact tells us that in order to develop a strategy which comes close to the optimal order $r(s+k)$(also see \cite{oymak2015simultaneously}), one has to deviate to other recovery methods. The authors of  \cite{oymak2015simultaneously} provide a few examples of minimization problems which recover solutions from an, up to log-factors, optimal amount of measurements, but they are not computationally feasible. 

The 'single-structure bottleneck' motivates why we in this paper choose to restrict the analysis to the program $\calP_{1,2}$ - the inclusion of the nuclear norm term does not fundamentally change the performance, while it does make the analysis of the algorithm significantly more complicated. \newline

An approach, which was proposed in \cite{lee2013near}, is to use an adapted version of the method of \emph{power factorization}, introduced in \cite{haldar2009rank}. 
Power factorization is fundamentally different from $\ell_{1,2}$-minimization, as it does not rely on minimizing one single convex function. 
The main idea is to exploit the fact that any rank-$1$-matrix can be written as $h_0z_0^*$ explicitely: After initializing $\hat{h}$ and $\hat{z}$, one iteratively fixes one of the vector and then solves a least-square problem for the second:
\begin{align*}
	\hat{h}^k &= \argmin \norm{\calA(\hat{h}(\hat{z}^{k-1})^*) -b }_2 \\
	\hat{z}^k &= \argmin \norm{\calA(\hat{h}^{k-1}\hat{z}^*)-b}_2.
\end{align*}
In the column-sparse case, one incorporates the sparsity assumption when solving for $\hat{z}^k$. The method the authors of \cite{lee2013near} choose is hard thresholding, but they state that other methods for sparse recovery could also be used for this purpose,  for instance $CoSaMP$. They are able to prove that if one very carefully initializes the two iterates $\hat{h}$ and $\hat{z}$, the method converges with high probability to the correct solution when using $\geqsim (s+k)\log(\max(n/s,k))$ measurements, which is essentially optimal. The initialization procedure used is however not computationally feasible. Because of this reason, the paper also provides an analysis for a more realistic initialization technique. The latter works under the same assumptions provided the $\infty$-norms of the vectors $h_0$ and $z_0$ are not too small (hence, they cannot be ''well spread'' on their respective supports). For general signals, they prove that $\geqsim sk\log(n)$ measurements suffice.

During the final preparations of this paper, an updated version \cite{lee2013near} was published, where an algorithm for matrices with rank higher than $1$ was discussed. The idea is again to utilize the singular value decomposition $Z= U\Lambda V^*$ and solve alternating minimization problems. The authors prove that under some technical assumptions, any matrix $Z_0$ with singular value decomposition $U_0 \Lambda_0 V_0^*$ with $U \in \R^{k,r}$ and $V\in \R^{r,n}$, where $U$ is row-$s_1$-sparse and $V$ is row-$s_2$-sparse satisfying some additional technical conditions (analogous to the $\norm{\cdot}_\infty$-bounds from above) can with high probability be recovered with $Cr(s_1+s_2)\log(\max(ek/s_1, en/s_2)$ Gaussian measurements.

	\subsection{Outline}
	The remainder of this paper is organized as follows: In Section \ref{sec:prel}, we present notation, clarifies the context and and revise duality theory of convex optimization. The latter plays a crucial role in Section \ref{sec:main}, where we present and prove the main findings of this work: We will provide results concerning all three aspects of soft recovery listed above. In particular, we will present a condition on $\calA$ securing that angular distances between the columns of the solution of $\calP_{1,2}$ and the ground truth signals are small, and provide upper bounds on how many Gaussian measurements are needed to secure that condition with high probability. In the final Section \ref{sec:heur}, we present two heuristic algorithms (called $NAST$ and Column Streamlining) for recovery of column-sparse, low rank matrices, and test their performance numerically.

% Notation very important: z.H

\section{Preliminaries} \label{sec:prel}

Let us begin by making the brief problem description from above a bit more precise, and introduce the notation which will be used throughout the paper. We consider a matrix $Z_0 \in \R^{n,k}$ (the ground truth signal) which is \emph{column-sparse}, i.e. the set
\begin{align*}
	\supp Z_0 = \set{i \ \in [n] : Z_0(i) \neq 0}
\end{align*}
is small. $[n]$ is a shorthand for the set $\set{1, 2, \dots n}$ and $Z_0(i)$ denotes the $i$:th column of $Z_0$. The task is to recover the matrix from $m$ linear measurements, given by the map $\calA : \R^{n,k} \to \R^m$.

We will most often view the matrix $Z_0$ as a collection of $n$ vectors  $(Z_0(i))_{i \in [n]}$ in $\R^k$. Since we will make claims about the direction of these vectors, it will be convenient to decompose each of those vectors $Z(i)$ into a direction $h_i \in \sph^{k-1}$ and a magnitude $z_i \in \R^{+}$. With this decomposition in mind, we write for $z \in \R^{n}_+$ and $H \in (\sph^{k-1})^n$
\begin{align*}
	z.H = \sum_{i \in [n]} z_i h_i e_i^* = \begin{bmatrix}
	z_1 h_1 \ z_2 h_2 \ \dots \ z_n h_n\end{bmatrix}.
\end{align*}
Note that if an element $z_k$ of $z$ is equal to zero, $z.H$ is well-defined even if the corresponding $h_k$ is not specified. Also note that due to the fact that we have agreed to the convention hat $z\geq 0$, the composition $Z \to z.H$ is unique up to the fact that a column $h_i$ can be choosen arbitrarily when $z_i=0$.

Dual to the decomposition $Z \mapsto z.H$, we decompose the map $\calA$ into the $n$ maps $A_i: \R^k \to \R^m$ through
\begin{align*}
	A_i h = \calA(h e_i^*).
\end{align*}
Note that then $\calA(z.H) = \sum_{i \in [n]} z_i A_i h_i$. We will sometimes also need the following induced map
\begin{align*}
	A_H: \R^n \to \R^m , z \mapsto \calA(z.H).
\end{align*}
Note that the matrix representation of $A_H$ is given by the matrix whose $i$:th column is given by $A_i h_i$. It is also not hard to convince oneself that the dual operator $A_H^*$ is given by
\begin{align*}
	A_H^* : \R^m \to \R^n, p \mapsto \left(\sprod{h_i, A_i^* p}\right)_{i \in [n]}
\end{align*}

We have already defined the $\ell_{1,2}$-norm. We will also need the $\ell_{\infty,2}$-norm:
\begin{align*}
	\norm{Z}_{\infty,2}= \max_{i \in [n]} \norm{Z(i)}_2
\end{align*}
Note that $\norm{\cdot}_{1,2}$ harmonizes particularly well with our decomposition $Z \to z.H$; we have $\norm{Z}_{1,2} = \norm{z}_1 = \sum_{i \in [n]} z_i$ (this is true also for the $\ell_{\infty,2}$-norm, we will however never use that statement). 

We will canonically measure distances between vectors $h, \hat{h} \in \sph^{k-1}$ by their angular (geodesic) distance
\begin{align*}
	\omega(h, \hat{h}) = \arccos( \langle{h,\hat{h}}\rangle).
\end{align*}
$\omega$ defines a metric on $\sph^{k-1}$. In particular, the triangular inequality holds:
\begin{align*}
	\omega( h_1, h_2) \leq \omega(h_1, h_3) + \omega(h_3,h_2).
\end{align*}

With this notation, we can formulate our main task as follows: Provide as weak conditions as possible guaranteeing that: ($Z_0=z^0.H^0$ is the ground truth signal, $\widehat{Z}=\hat{z}.\widehat{H}$ is a solution of $\calP_{1,2}$)

\begin{itemize}
	\item $\omega(h_i^0, \hat{h}_i) \leq \alpha$ for $i \in S=\supp Z_0$ (for some previously specified, small $\alpha)$.
	
	\item $\norm{\hat{z} \vert_{S^c} }_1 \leq \epsilon$, (for some previously specified, small $\epsilon$).
	
	\end{itemize}

%Exact recovery

\subsection{Duality in Convex Programming and Exact Recovery.}

Before considering the task of tackling the problem of soft recovery, let us deploy some standard techniques for formulating a condition guaranteeing that the ground truth signal $Z_0$ is the unique solution of $\calP_{1,2}$. This will provide us with important insights of the structure of the $\ell_{1,2}$-norm, give us the opportunity to review basic facts about convex optimization, as well provide us with a benchmark that our soft recovery condition has to beat.

In the analysis, we will use \emph{duality}. Given a convex program  of the form
\begin{align}
	\min f(x) \text{ subject to } Ax = b, \ Dx \geq d \tag{$\calP$}
\end{align}
where $f: \R^q \to \R$, $A \in \R^{\ell,q}$ $D \in \R^{\kappa, q}$, it is defined as follows: First, consider the Lagrangian
\begin{align*}
	\calL : \R^q \times \R^\ell \times \R_{+}^\kappa \to \R, (x, \lambda, \mu) \mapsto f(x) + \sprod{\lambda, b-Ax} + \sprod{\mu, d-Dx}.
\end{align*}
With the help of the Lagrangian, the dual function $g: \R^\ell \times \R^\kappa_+\to \R \cup \set{-\infty} $ is defined through
\begin{align*}
	g(\lambda, \mu) := \inf_{x \in \R^q} \calL(x,\lambda, \mu).
\end{align*}
The dual problem is then defined as
\begin{align}
	\max_{\lambda, \mu \geq 0} g(\lambda, \mu). \tag{$\calD$}
\end{align}
The relation between the primal and dual problem is as follows: In any case (and in much greater generality than was presented here), the optimal value $d^*$ of the dual problem is not greater than the optimal value of the primal problem $p^*$
\begin{align*}
	p^* \geq d^*.
\end{align*}
This is known as \emph{weak duality.} Under quite general conditions, for instance for problems where  $f$ is convex, the constraints are linear , and that there exists an $x$ with $b=Ax$ and $Dx \geq d$, we even have \emph{strong duality}
\begin{align*}
	p^* =d^*.
\end{align*}
 For more information about duality in convex optimization, we refer to the book \cite{ConvexOpt}.

Let us now state and prove the exact recovery condition.
\begin{prop} \label{prop:exactRecovery}
	Let $Z_0 = z^0.H^0$ be supported on the set $S$ and $\calA: \R^{n,k} \to \R^m$ be given. $Z_0$ is a solution to the program $\calP_{1,2}$ if and only if there exists a $p \in \R^m$ such that $V= \calA^*p$ satisfies
	\begin{align}
		V(i) = h_i^0, \quad &i \in S  \label{eq:exactCond1}\\
		\norm{V(i)}_2 \leq 1, \quad& i \notin S \label{eq:exactCond2}
	\end{align}
\end{prop}
\begin{proof}
 	Let us begin by calculating the dual problem $\calD_{1,2}$. The Lagrange function is given by
 	\begin{align*}
 		\calL : \R^{n,k} \times \R^m \to \R , (Z,p) \mapsto \norm{Z}_{1,2}  + \sprod{p,b-\calA(Z)} = \sprod{p,b} + \sum_{i \in n} z^0_i(1- \sprod{A_i^*p, h_i^0}).
 	\end{align*}
 	From this formula, it is easily seen that the dual function $g$ is given by 	\begin{align*}
 		g(p) = \begin{cases} \sprod{p, b} \text{ if } \norm{(A_i^*p)_{i\in [n]}}_{\infty,2} \leq 1 \\
 		- \infty \text{ else.}\end{cases}
 	\end{align*}
 	Due to strong duality, $Z_0$ is a solution to $\calP_{1,2}$ if and only if the optimal value of the problem
 	\begin{align*}
 		\max \sprod{p,b} \text{ subject to } \forall i :  \norm{A_i^*p}_2 \leq 1
 	\end{align*}
 	is given by $\norm{Z_0}_{1,2}= \Vert z^0 \Vert_1$ . Due to the compactness of the set of feasible vectors, this is the case if and only if there exists a $p \in \R^m$ with $\norm{A_i^*p}_2 \leq 1$ and $\sprod{p,b}=\norm{z^0}_1$. It is however clear that in this case,
 	\begin{align*}
 		\Vert z^0 \Vert_1 = \sprod{p, b} = \sum_{i \in S} z_i^0 \sprod{p, A_i h_i^0} = \sum_{i \in S } z_i^0 \sprod{A_i^*p, h_i^0}
 	\end{align*}
 	if and only if $A_i^*p =h_i^0$. Noting that $(\calA^*p)(i)=A_i^*p$, we see that the claim has been proven.
\end{proof}

% Lemma: "inf <y,p>='' ... Strong duality important! (MOST IMPORTAMENT TOURNAMENT! )

We will tackle also the soft recovery problem using the tools of duality, however not as directly as above. The main idea will be to view the $\ell_{1,2}$-minimization as a family of $\ell_1$-minimizations . To be precise, let us assume that the solution of $\calP_{1,2}$ is given by $\hat{z}.\widehat{H}$. It is then not hard to convince oneself that $\hat{z}$ is the solution of the  problem 
\begin{align}
	\norm{ z}_1 = \sum_{i\in [n] } z_i\text{ subject to } A_{\widehat{H}}z =b, z \geq 0  \tag{$ \calP_1^{+}(\widehat{H})$}
\end{align}
After all, for any other $z \geq 0$ with $A_{\widehat{H}}z=b$, we have $\calA( z. \widehat{H})=b$ and, due to the optimality of $\hat{z}. \widehat{H}$, $\norm{\hat{z}}_1 = \Vert{\hat{z}. \widehat{H}}\Vert_{1,2} \leq \Vert {z. \widehat{H}}\Vert_{1,2} = \norm{z}_1$. Using this relatively elementary observation, we can prove the following property of the minimizer $\hat{z}.\widehat{H}$
\begin{lem}\label{lem:AHopt}
Let $\hat{z}.\widehat{H}$ be the minimizer of $\calP_{1,2}$ for $b = \calA(z^0. H^0)$, where $\norm{z^0}_1=1$. Then 
\begin{align*} 
	 \min_{\sprod{b,p} \geq 1} \max_{i\in [n]} \langle{\hat{h}_i, A_i^*p}\rangle  \geq 1
\end{align*}
\end{lem}

\begin{proof}
	The dual problem of $\calP_{1}^+(\widehat{H})$ can be written in the following way
	\begin{align}
		\max \sprod{p, b} \text{ subject to } \max_{i \in [n]} (A_{\widehat{H}}^*p)_i  \leq 1. \tag{$\calD^{+}_{1}(\widehat{H})$}
	\end{align}
	(The, relatively standard, calculations leading up to this formulation of the dual problem are, for completeness, presented in Lemma \ref{lem:DualPosL1} in the Appendix.) There exists a $\hat{z} \geq 0 $ so that $\hat{z}.\widehat{H}$ solves $\calP_{1,2}$, i.e. in particular $b= \calA(\hat{z}. \widehat{H}) = A_{\widehat{H}}\hat{z} =b$. Hence, strong duality holds. Therefore, the optimal value of $\calD_{1} (\widehat{H})$ is equal to $\Vert \hat{z} \Vert_1$.

	Now, due to the optimality of $\hat{z}.\widehat{H}$,	
	\begin{align*}
		\Vert \hat{z} \Vert_1 = \Vert \hat{z}. \widehat{H} \Vert_{1,2} \leq \Vert z^0 . H^0 \Vert_{1,2} = \Vert z^0 \Vert_1 =1.
	\end{align*}
	This means that the optimal value of $\calD_1(\widehat{H})$ is not larger than one, which implies that there cannot exist a $p$ with $\max_{i \in [n]}\sprod{h_i, A_i^*p}= \max_{i \in n} (A^*_{\widehat{H}}p)_i < 1$ and $\sprod{p,b}\geq 1$. If that would be the case,  a renormalized version of $p$ would satisfy $\max_{i \in [n]} (A^*_{\widehat{H}}p)_i \leq 1$ and $\sprod{b,p}>1$, which would contradict the fact that the optimal value of $\calD_1(\widehat{H})$ is not larger than 1.  Hence, the claim has been proven.
\end{proof}

\section{Main Results} \label{sec:main}

With Lemma \ref{lem:AHopt} in our toolbox, we are ready to state and prove one of the main result of this paper. It provides a condition guaranteeing that the angular distances between the columns of the solution of $\calP_{1,2}$ and the ground truth signals are small.

\begin{theo} \label{th:softRecAngle}
	Let $Z_0=z_0.H_0$ be supported on the set $S$ , $i^*\in S$ and $\alpha>0$ be fixed. If there exists a vector $p^* \in \R^m$  so that $V=\calA^*p$ has the following properties
	\begin{align}	
			\sum_{i \in S} \sprod{h_i^0, V(i)} z^0_i &\geq \norm{z^0}_1 \label{eq:SoftCond1}\\
			\norm{V(i)}_2 &< 1 \quad i\neq i^*\label{eq:SoftCond2} \\
			\omega\left(h_{i^*}^0, \tfrac{V(i^*)}{\norm{V(i^*)}_2}\right) &\leq \alpha.\label{eq:SoftCond3} \\
			\norm{V(i^*)}_2 \cos \alpha &\leq 1 \label{eq:SoftCond4}
	\end{align}
Then any minimizer $\hat{z}.\widehat{H}$ of $\calP_{1,2}$ obeys
\begin{align*}
	\omega(\hat{h}_{i^*},h_{i^*}^0)\leq 2 \alpha.
\end{align*}
\end{theo}

\begin{proof}
	Let us begin by noting that we without loss of generality may assume that $\Vert z^0 \Vert_1 =1$. This since only the norms of the columns of the minimizer changes when we scale $z^0$, and not the directions of them.
	
	We have, due to \eqref{eq:SoftCond1}
	\begin{align*}
		\sprod{b,p^*} = \langle\calA(z^0.H^0),p^*\rangle = \sum_{i\in S} z_i^0 \sprod{A_ih_i^0,p^*}  \geq 1.
	\end{align*}
	This, together with Lemma \ref{lem:AHopt} implies that there exists an $i \in [n]$ with $$\langle{A_i^*p^*,\hat{h}_i} \rangle \geq 1.$$ Due to \eqref{eq:SoftCond2}, this index $i$ must be equal to $i^*$. 
	
	Now suppose that  $\omega(\hat{h}_{i^*},h_{i^*}^0)$  is strictly larger than $2 \alpha$. Then by \eqref{eq:SoftCond3} and the triangle inequality of $\omega$
	\begin{align*}
		 \omega(\hat{h}_{i^*}, \tfrac{A_i^*p^*}{\norm{A_i^*p^*}})  \geq \omega(\hat{h}_{i^*},h^0_{i^*}) -\omega\left(h_{i^*}^0, \tfrac{A_{i^*}p^*}{\norm{A_{i^*}p^*}_2}\right) > 2\alpha - \alpha =\alpha ,
	\end{align*}
	which implies
	\begin{align*}
	\langle{A_{i^*}^*p^*,\hat{h}_{i^*}}\rangle= \norm{A_{i^*}^*p^*}_2 \cos( \omega(\hat{h}_{i^*}, \tfrac{A_{i^*}^*p^*}{\norm{A_{i^*}^*p^*}})) \leq \norm{A_{i^*}^*p^*}_2 \cos(\alpha) \leq 1 
	\end{align*}
	due to \eqref{eq:SoftCond4}. This is a contradiction and the proof is finished.
\end{proof}

% Column-stability, dual certificates
% Statistical dimensions of the dual certificate conditions. 'Max is enough'

Let us make some comments on \eqref{eq:SoftCond1}-\eqref{eq:SoftCond4}.

 First, the set defined by the four constraints is convex for $\alpha\leq \tfrac{\pi}{2}$. First, it is easily seen that both $\eqref{eq:SoftCond1}, \eqref{eq:SoftCond2}$ and $\eqref{eq:SoftCond4}$ define convex sets. As for $\eqref{eq:SoftCond3}$, notice that it can be rewritten as 
\begin{align*}
	\langle A_{i^*}^*p,h_{i^*}^0 \rangle \geq \cos(\alpha) \norm{A_i^*p}_2
\end{align*}
If $\alpha \leq \tfrac{\pi}{2}$, $\cos(\alpha)\geq 0$ and hence for $p_1,p_2$ satisfying \eqref{eq:SoftCond3} and $\theta \in (0,1)$, we have
\begin{align*}
	\norm{A_{i^*}^*(\theta p_1 +(1-\theta)p_2)}_2 \cos \alpha &\leq \theta \norm{A_{i^*}^*p_1}_2 \cos \alpha + (1-\theta) \norm{A_{i^*}p_2}_2 \cos \alpha  \\
	&\leq \theta \sprod{A_{i^*}p_1,h_{i^*}^0} + (1-\theta) \sprod{A_{i^*}p_2,h_{i^*}^0} = \sprod{A_{i^*}^*(\theta p_1 + (1-\theta) p_2, h_{i^*}^0}.
\end{align*}

Second, the condition described in Theorem \ref{th:softRecAngle} is in essence weaker than the one described in Proposition \ref{prop:exactRecovery}, at least when dealing with reasonable random matrices. To see this, let us first note that for such matrices, it does not matter if we replace the strict inequality sign in \eqref{eq:SoftCond2} with a ''$\leq$''- this does not change for instance the Gaussian width \cite{Gordon1988} or the statistical dimension \cite{AmelunxLotzMcCoyTropp2014} of the set (see also below). With this slight change, we see that a vector satisfying \eqref{eq:exactCond1} and \eqref{eq:exactCond2} necessarily obeys
\begin{align*}
	\sum_{i \in S} \sprod{h_i^0, A_i^*p^*}z_i^0 &= \sum_{i \in S} \sprod{h_i^0,h_i^0}z_i^0 = \norm{z^0}_1 \\
	\norm{A_i^*p}_2 &\leq 1, \quad i \neq i^* \\
	\omega(h_{i^*}^0, \tfrac{A_{i^*}^*p^*}{\norm{A_{i^*}^*p^*}} ) &= \omega(h_{i^*}^0,h_{i^*}^0)=0\leq \alpha \\
	\norm{A_{i^*}p}_2 \cos(\alpha) &= \norm{h_{i^*}^0}_2\cos(\alpha)= \cos(\alpha) \leq 1.
\end{align*}
Hence, if the exact recovery statement is fulfilled, the soft recovery statement (in essence) is also.

Third, it is not a big problem that we  for every $i \in S$  need to (separately) ensure the existence of one vector in the range of $\calA^*$ fulfilling \ref{eq:SoftCond1}-\eqref{eq:SoftCond4} to ensure that $\omega(h_i^0, \hat{h}_i^0)$ for all $i \in S$. To argue why this is the case, consider this (slightly reformulated) inequality from \cite[Theorem 7.1]{AmelunxLotzMcCoyTropp2014}: for a convex cone $C \sse \R^d$, we have for an $m$-dimensional uniformly distributed random subspace $V$
\begin{align}
	m \geq d - \delta(C) + \sqrt{4d\log\eta^{-1}}\impl \prb{C \cap V \neq {0})}\leq  \eta, \label{eq:statDimProp}
\end{align}
where $\delta(C)$ denotes the \emph{statistical dimension} of the cone $C$. Due to the linear structure of $\ran \calA^*$, there exists a vector in $\ran \calA^*$ fulfilling \ref{eq:SoftCond1}-\ref{eq:SoftCond4} exactly when $\ran \calA^* \cap \calC_{i^*} \neq \set{0}$, where  $\calC_i$ is the cone generated by the constraints \eqref{eq:SoftCond1}-\eqref{eq:SoftCond4}.
\eqref{eq:statDimProp} now tells us that if $V$ is a uniformly distributed random subspace of dimension $m$ in $\R^{k,n}$
\begin{align*}
\prb{ \forall i \in S : \calC_i \cap V \neq {0})} &= 1 - \prb{\exists i \in S : \calC_i \cap V ={0})} \geq 1 - \sum_{i \in S} \prb{\calC_i \cap V ={0}}  \\
&\geq 1 - s \max_{i \in S} \prb{\calC_i \cap V ={0}}  \geq 1 - \eta
\end{align*}
if $m \geq nk - \min_i \delta(\calC_i) + \sqrt{4d\log((s\eta)^{-1})}$. In the case that $\calA$ is a random Gaussian, $\ran \calA$ will have exactly the distribution of $V$ above. Hence, the number of Gaussian measurements needed to secure the existence of a $p$ satisfying \eqref{eq:SoftCond1}-\eqref{eq:SoftCond4} for all $i \in S$ is basically the same amount as the one needed to secure one satisfying the ''hardest'' constraint. The only difference is a factor of $\log$-type.

\subsection{The Statistical Dimension of the Cone Generated by \eqref{eq:SoftCond1} - \eqref{eq:SoftCond4}.} \label{sec:statDim}

We have already discussed the concept of \emph{statistical dimension} $ \delta(C)$ of a convex cone $C \sse \R^d$. It was introduced in \cite{AmelunxLotzMcCoyTropp2014}. The importance of the concept is captured in the formula \eqref{eq:statDimProp} - the statistical dimension provides a threshold amount of Gaussian measurements needed for $\ran A$ to intersect the cone. It can be defined in many ways. The most convenient way is probably to first define it for closed convex cones $C$ as 
\begin{align*}
	 \delta(C) = \erw{\norm{\Pi_C g}_2^2},
\end{align*}
where $\Pi_C$ denotes the (non-linear) orthogonal projection onto $C$, and $g \in \R^d$ a Gaussian vector. For general convex cones $K$, we define $\delta(K)$ as the statistical dimension of the closure of $K$.

  It is hard to give an exact expression for the statistical dimension  of the cone generated by \eqref{eq:SoftCond1} - \eqref{eq:SoftCond4} (let us denote it by $\calC_{i^*}$, as we did above).
In this section, we will prove a lower bound of $\delta (\calC_{i^*})$, which through \eqref{eq:statDimProp} provides an upper bound on the amount of measurements needed to guarantee soft recovery. This bound is also not a closed expression, but one that easily can be calculated numerically. The first step of the argumentation will be to identify a subset of $\calC_{i^*}$ which is a convex cone whose statistical dimension is easier to handle. Let us begin by describing the subset.
\begin{lem}
Let $\alpha < \tfrac{\pi}{2}$. Then, $\calC_{i^*}$ contains $\cone \calM_{i^*}$, where $\calM_{i^*} \sse \R^{k,n}$ is given by the equations
\begin{align*}
	\norm{V_i}_2 &< 1 \quad i \neq i^* \\
	\sprod{V_i, h_i^0} &\geq \sigma \quad i \in S \backslash \set{i^*} \\
	\norm{V_{i^*}}_2  &\leq \tfrac{1}{\cos \alpha} \quad  \\
	\sprod{V_{i^*},h_{i^*}^0} &\geq \tfrac{\sigma}{\cos \alpha}.
\end{align*}
$\sigma$ is a parameter defined through
\begin{align*}
	\sigma := \frac{1}{1+(\tfrac{1}{\cos(\alpha)} -1 ) \tfrac{z_i^0}{\norm{z^0}_1}}.
\end{align*}
\end{lem}

\begin{proof}
	 Let $V \in \calM_{i^*}$. Then trivally \eqref{eq:SoftCond2} and \eqref{eq:SoftCond4} are satisfied. As for \eqref{eq:SoftCond1}, we have
	 \begin{align*}
	 	\sum_{i \in S} \sprod{h_i^0, V_i}z_i^0 &\geq z_{i^*}^0 \cdot \tfrac{\sigma}{\cos \alpha} + \sum_{i \in S\backslash{i^*}} \sigma z_{i^*}^0   = z_{i^*}^0 \cdot \tfrac{\sigma}{\cos \alpha} + \sigma (\norm{z^0}_1 -z_{i^*}^0) = \norm{z^0}_1 .
\end{align*}	
As for \eqref{eq:SoftCond3}, we need to prove that $\sprod{V_{i^*},h_{i^*}^0} \geq  \norm{V_{i^*}}_2 \cos \alpha$. For this, it suffices to show that $\sigma \geq \cos \alpha$, since we already know that $\sprod{V_{i^*},h_{i^*}^0} \geq \tfrac{\sigma}{\cos \alpha}$ and $\norm{V_{i^*}}_2 \leq \tfrac{1}{\cos \alpha}$. This is however not hard: $\sigma \geq \cos \alpha$ is equivalent to
\begin{align*}
1 \geq \cos\alpha +\tfrac{z_i^0}{\norm{z^0}_1} ( 1- \cos \alpha) = \tfrac{z_i^0}{\norm{z^0}_1} + \cos \alpha  \left(1-\tfrac{z_i^0}{\norm{z^0}_1}\right),
\end{align*}
which is true, since $\cos \alpha \leq 1$.

\end{proof}

With the help of the previous lemma and the definition of the statistical dimension of a cone, we derive the following lower bound for $\delta(\calC_{i^*})$:
\begin{align}
	\delta(\calC_{i^*}) \geq \delta( \cone \calM_{i^*} ) = \erw{\norm{\Pi_{\cone \calM_i}G}_2^2} = nk - \erw{ \inf_{\tau >0} \dist(G - \tau \calM_{i^*})^2} \geq nk - \inf_{\tau>0} \erw{ \dist(G - \tau \calM_{i^*})^2}, \label{eq:CLowBound}
\end{align}
where $G \in \R^{k,n}$ is Gaussian. We used the Lemma of Fatou and the following trick from \cite{AmelunxLotzMcCoyTropp2014}: for $g \in \R^d$ Gaussian and $C \sse \R^d$, we have
\begin{align*}
	\erw{\norm{\Pi_C g}_2^2}= \erw{ \norm{g}_2^2 - \min_{c \in C} \norm{g-c}_2^2} = d- \erw{\min_{c \in C} \norm{g-c}_2^2}.
\end{align*}
In other words, the task of bounding $\delta(\calM_{i^*})$ from below can be accomplished by bounding $\inf_{\tau>0} \erw{ \dist(G - \tau \calM_{i^*})^2}$ above. In order to solve this task, let us first have a look at the function  $V \to \dist(V - \tau \calM_{i^*})^2$.

\begin{lem} \label{lem:Mproj}
	Define $\beta$ through $\sigma = \cos(\beta)$ and let $V \in \R^{k,n}$ Then
	\begin{align*}
		\dist(V - \tau \calM_{i^*})^2 = \sum_{i \in [n]} d_i(V_i),
	\end{align*}
	where $d_i: \R^n \to \R$ are defined through
	\begin{align*}
		d_i(v) &= \pos(\norm{v}_2-\tau)^2  &  i \notin S \\
		d_i(v) &= \begin{cases} (\langle v,h_i^0\rangle-\tau \cos \beta)^2 + \pos(\norm{\Pi_{\sprod{h_i^0}^\perp} v }_2 - \tau \sin \beta)^2 &\text{ if } \norm{\Pi_{\sprod{h_i^0}^\perp}v}_2 \geq \tan \beta \langle h_i^0, v\rangle \\
		&\text{ or } \sprod{v, h_i^0}\leq \tau \cos \beta	\\
		\pos(\norm{v}_2 - \tau)^2	 & \text{ else.}
		\end{cases} & i \in S \backslash \set{i^*} \\
		d_{i^*}(v) &= \begin{cases} (\langle v,h_i^0\rangle-\tau \tfrac{\cos \beta}{\cos{\alpha}})^2 + \pos(\norm{\Pi_{\sprod{h_i^0}^\perp} v }_2 - \tau \tfrac{\sin \beta}{\cos \alpha})^2 &\text{ if } \norm{\Pi_{\sprod{h_i^0}^\perp}v}_2 \geq \tan \beta \langle h_i^0, v\rangle \\
		&\text{ or } \sprod{v, h_i^0}\leq \tau\tfrac{\cos \beta}{\cos \alpha}	\\
		\pos(\norm{v}_2 - \tfrac{\tau}{\cos \alpha})^2	 & \text{ else.}
		\end{cases}
		\end{align*}
\end{lem}

\begin{proof}
	First, it is immediately clear that 
	\begin{align*}
	\dist(V - \tau \calM_{i^*})^2 = \min_{\substack{\norm{w}_2 \leq \tfrac{\tau}{\cos \alpha} \\ \sprod{w, h_i^0}\geq \tau \tfrac{\cos \beta }{\cos \alpha}}} \norm{v_{i*} -w}_2^2+ \sum_{ i \in S \backslash \set{i^*}} \min_{\substack{\norm{w}_2 \leq \tau \\ \sprod{w, h_i^0}\geq \tau \cos \beta }} \norm{v_i -w}_2^2 + \sum_{i \notin S} \min_{\norm{w}_2 \leq \tau} \norm{v_i -w}_2^2 .
	\end{align*}
	It is also easy to see that $\min_{\norm{w}_2 \leq \tau} \norm{v_i -w}_2^2 = \pos(\norm{v_i}_2-\tau)^2$ (we can choose $w$ parallel to $v_i$ and exhaust as much of its length as possible - in particular its whole length if $\norm{v_i}_2 \leq \tau$). For the other expressions, we have to be a little more careful. We will only treat the case that $i \in S \backslash \set{i^*}$ - the case $i = i^*$ is analogous (the only difference are extra $\tfrac{1}{\cos \alpha}$-factors appearing everywhere, whence we choose to omit that case).
	
		To calculate $$\min_{\substack{\norm{w}_2 \leq \tau \\ \sprod{w, h_i^0}\geq \tau \cos \beta }} \norm{v -w}_2^2,$$ we distinguish two cases. A graphical depiction of the argumentation in each case can be found in Figure \ref{fig:MProj}
	\begin{figure}
		
		\centering
		\includegraphics[scale=.3]{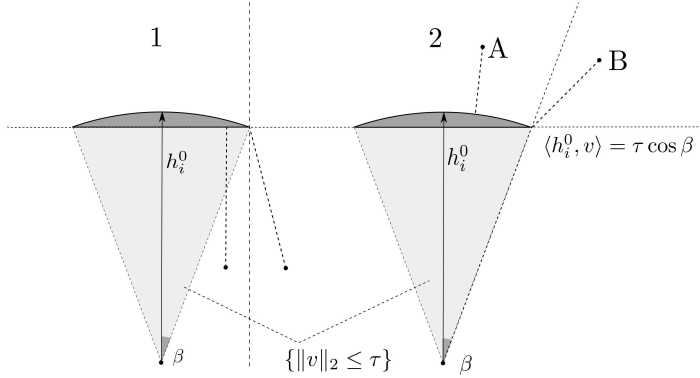}
		\caption{\label{fig:MProj} Graphical depiction of the argumentation in the Proof of Lemma \ref{lem:Mproj}}.
	\end{figure}

	{\bf Case 1 : $\sprod{v, h_i^0}\leq \tau \cos(\beta)$.} In this case, the minimizer of $\norm{v-w}_2$ must lie on the $(k-1)$-dimensional disc  $$\set{w \ \vert \ \sprod{w,h_i^0}= \tau \cos \beta, \ \norm{ \Pi_{\sprod{h_i^0}^\perp} w}_2 \leq \tau \sin \beta}.$$ With this insight, it is clear that the minimum is given as claimed.
	
	{\bf Case 2:  $\sprod{v, h_i^0}\geq \tau \cos(\beta)$.} In this case, the minimizer must lie on the spherical cap $\norm{w}_2=1, \omega(w, h_i^0)\leq \beta$. This makes it clear that the minimum depends on the angular distance between $\tfrac{v}{\norm{v}_2}$ and $h_i^0$. 
	
	If the angle is smaller than $\beta$, we can choose $w$ parallel to $v$, and the smallest possible separation is equal to $\pos(\norm{v}_2 - \tau)$. 
	 In the other case, we may choose $w$ in the plane spanned by $h_i^0$ and $v$ with an angular distance to $h_i^0$ at most $\beta$. The minimizer is thus equal to $\tau (\cos \beta \ h_i^0 + \sin \beta \ \Pi_{\sprod{h_i^0}^\perp}v)$, and the minimum equal to the claimed expression.
\end{proof}

From the proof of the last lemma, it is evident that we can estimate for $i \in S \backslash \set{i^*}$
\begin{align} \label{eq:diEst}
	d_i(v) \leq (\langle v,h_i^0\rangle-\tau \cos \beta)^2 + \pos(\norm{\Pi_{\sprod{h_i^0}^\perp} v }_2 - \tau \sin \beta)^2
\end{align}
This may seem crude, but we in fact only make a non-strict estimate on the set $\set{v \vert \omega( \tfrac{v}{\norm{v}_2}, h_i^0)<\beta, \sprod{v,h_i^0} \geq \tau \cos(\beta)}$. This set has a very small Gaussian measure for small angles $\beta$ and large dimensions $k$. An analogous estimate can be made for $i=i^*$.

This final estimate allows us to prove the following proposition.

\begin{prop}
	Let $G\in \R^{k,n}$ be Gaussian. We have
	\begin{align*}
		\erw{\dist(G- \tau \calM_i)^2} \leq \Phi_{k,s,n,\alpha,\beta}(\tau)
		\end{align*}
		Where
		\begin{align*}
		\Phi_{k,s,n,\alpha,\beta}(\tau) := s + \tau^2 \tfrac{\cos^2\beta}{\cos^2 \alpha} + \erw{ \pos(\norm{g^{k-1}}_2 - \tau \tfrac{\sin \beta }{\cos\alpha})^2} &+ (s-1) \left( \tau^2 \cos^2\beta + \erw{ \pos(\norm{g^{k-1}}_2 - \tau \sin \beta )^2}\right) \\
		&+ (n-s) \erw{ \pos(\norm{g^k}_2-\tau)^2}
\end{align*}
where $g^d$ denotes an $d$-dimensional Gaussian and $s = \abs{S}$.

In particular, if $\calA : \R^{k,n} \to \R^m$ is Gaussian with $m \geq \inf_{\tau >0} \Phi_{k,s,n,\alpha,\beta}(\tau)$, $2\alpha$-soft recovery of the $i^*$-column is guaranteed with high probability.
\end{prop}

\begin{proof}
	We have due to Lemma \ref{lem:Mproj}
	\begin{align*}
	\erw{\dist(G- \tau \calM_i)^2} =  \sum_{i \in [n]} \erw{d_i(G_i)} = \erw{d_{i^*}(G_{i^*}) }+ \sum_{i \in S \backslash \set{i^*}} \erw{d_i(G_i)} + (n-s) \erw{ \pos(\norm{g^k}_2-\tau)^2},
	\end{align*}
	since each column of a Gaussian $G \in \R^{k,n}$ is a $k$-dimensional Gaussian. \eqref{eq:diEst} furthermore implies for $i \in S \backslash \set{i^*}$
	\begin{align*}
		\erw{d_i(G_i)} \leq  \erw{ \langle G_i,h_i^0\rangle^2} - 2 \tau \cos \beta \ \erw{\langle G_i, h_i^0 \rangle}+& \tau^2 \cos^2 \beta + \erw{ \pos(\norm{\Pi_{\sprod{h_i^0}^\perp} G_i }_2 - \tau \sin \beta)^2} \\
		& =1 + \tau^2 \cos^2 \beta + \erw{ \pos(\norm{g^{k-1}}_2 - \tau \sin \beta )^2},
	\end{align*}
	where in the last step we used that projections of Gaussians again are Gaussians. By performing a similar calculation for $\erw{d_{i^*}(G_{i^*})}$ and summing all of the terms, we arrive at the first statement.
	
	The second statement is now a trivial consequence of the first one together with the lower bound \eqref{eq:CLowBound} and the statement \eqref{eq:statDimProp}.
	
\end{proof}

% Lite numeriska experiment med detta
\begin{rem}
By putting $\alpha=0$ (which implies $\beta=0$), we arrive at a corresponding formula for the amount of measurements needed for exact recovery. Going through all of the calculations in this section again, one sees that in the case $\alpha=0$, all inequalities are in fact equalities, and we hence arrive at the exact value for the statistical dimension in this case. We may even simplify (since $\erw{ \pos(\norm{g^{k-1}}_2 - \tau \cdot 0 )^2}=k-1)$ 
\begin{align*}
	\Phi_{k,s,n,0,0}(\tau) = sk + s\tau^2 + (n-s) \erw{\pos(\norm{g^{k}}_2 -\tau)^2}.
\end{align*}
\end{rem}

Using computer software, one can calculate $m_{k,s,n,\alpha,\beta}=\inf_{\tau >0} \Phi_{k,s,n, \alpha, \beta}(\tau)$. Let us discuss some of the qualitative behaviours one can read out of these numerical evaluations.

\begin{figure}
\centering
	\includegraphics[scale=.3]{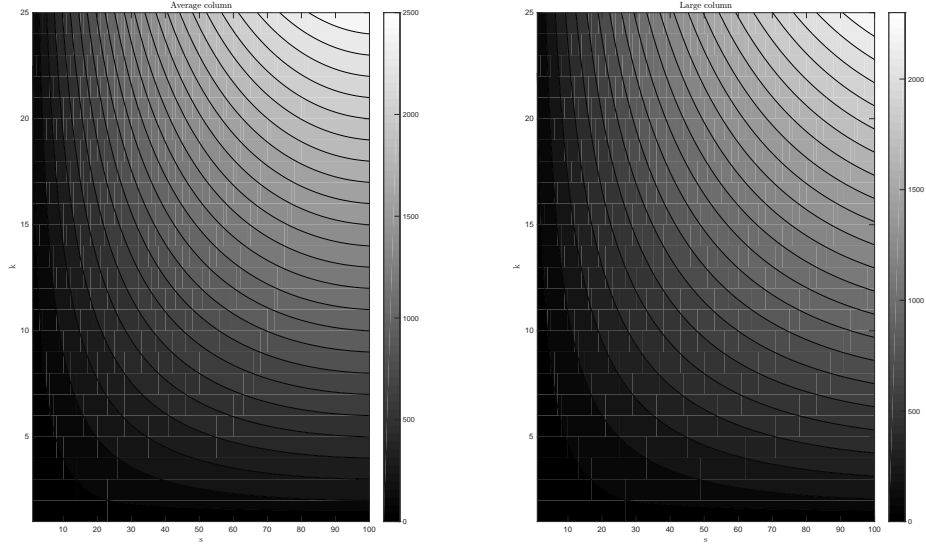}
	\caption{$m_{k,s,n,\alpha,\beta}$ depending on $k$ and $s$ for average (left) and large (right) column. $n=100$, $\alpha=\tfrac{\pi}{10}$. \label{fig:PhaseTransDiag}}
\end{figure}

\begin{figure}
\begin{minipage}[b]{.45\textwidth}

\centering
\includegraphics[scale=.31]{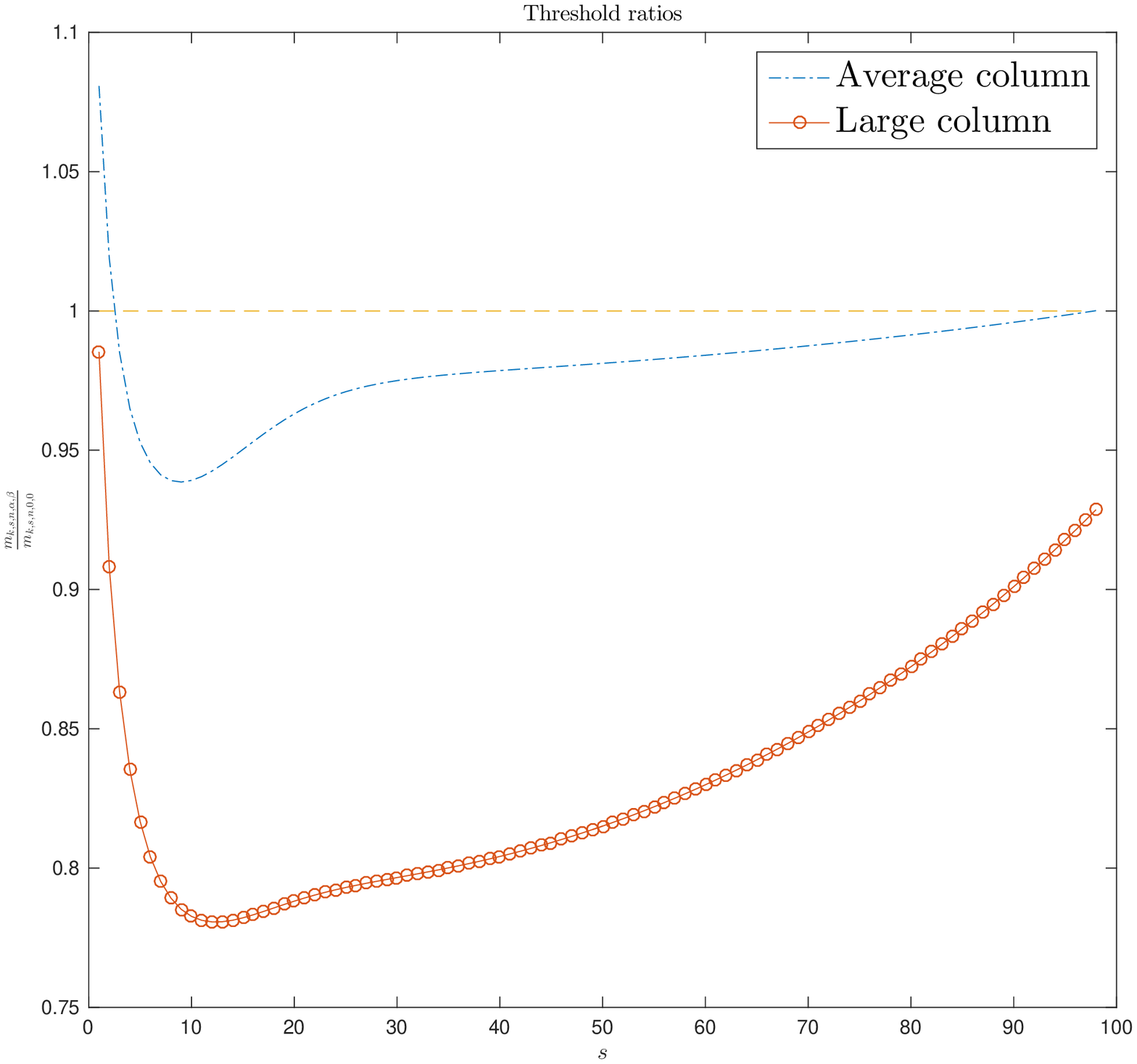}
\caption{The relation between $m_{k,s,n,\alpha,\beta}$ and $m_{k,s,n,0,0}$ for $k=10$, $n=100$, different values of $s$ and two types of columns. \label{fig:ThresholdRatios}}
\end{minipage} \quad
\begin{minipage}[b]{.45 \textwidth}

\centering \includegraphics[scale=.28]{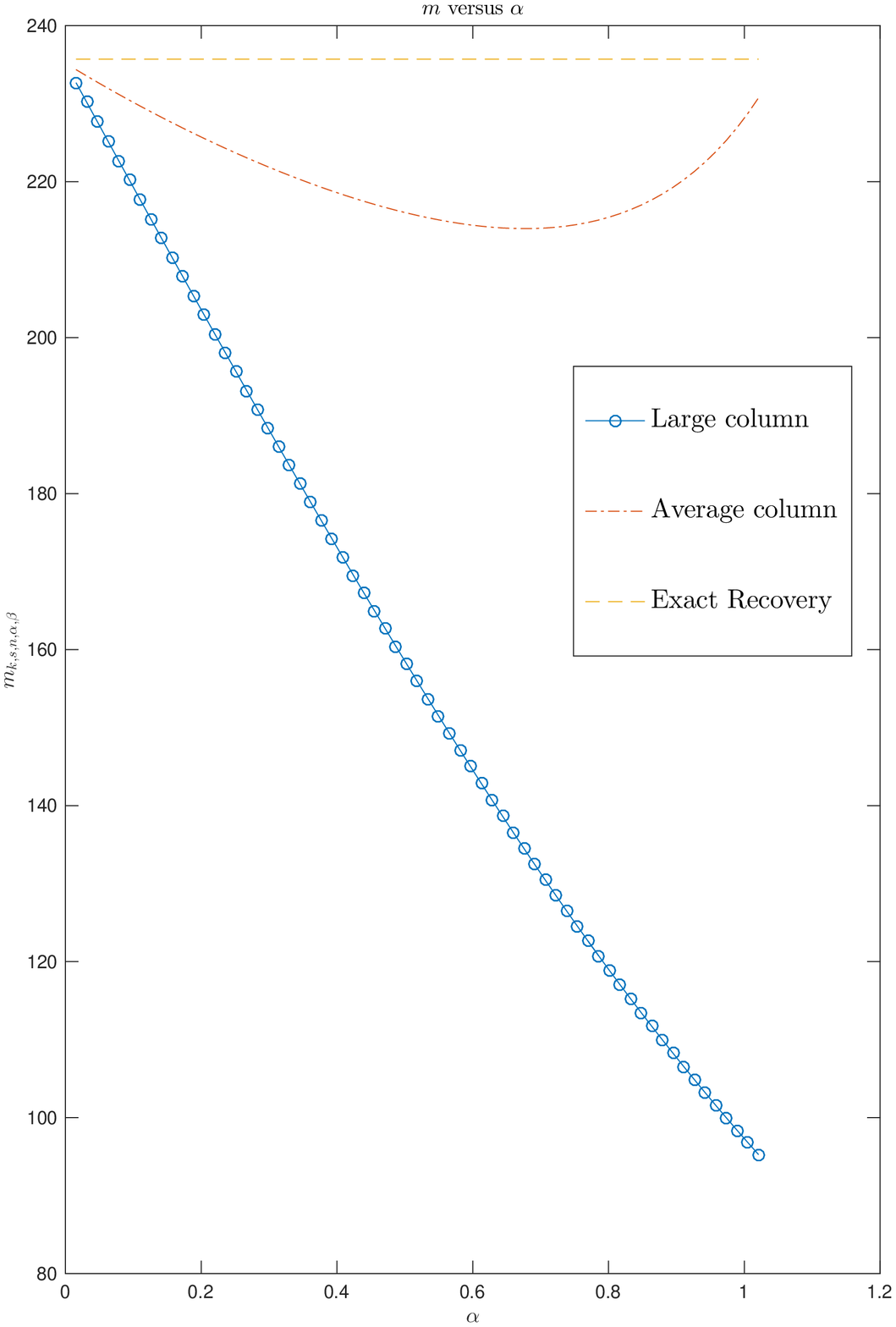}\caption{$m_{s,n,k,\alpha,\beta}$ for $k=s=10$, $n=100$ depending on $\alpha$ for two types of columns.\label{fig:angleReaction}}

\end{minipage} \end{figure}

 In Figure \ref{fig:PhaseTransDiag}, we fix $\alpha = \tfrac{\pi}{10}$, $n=100$ and  plot $m_{k,s,n,\alpha,\beta}$  for varying $k$ and $s$. We investigate two situations: In the left plot, we set $\tfrac{z^0(i^*)}{\norm{z^0}_1}=\tfrac{1}{s}$ -- this corresponds to  an ''average size''-column of $Z_0$. In the right, we instead fix $\tfrac{z^0(i^*)}{\norm{z^0}_1}=.9$ -- this corresponds to a large column. With the exception of $s=1$ (for which $1/s >.9$), $m_{k,s,n,\alpha,\beta}$ is larger for the average size column case than for the large column case, as is intuitively clear - it should be easier to recover (the more significant) large columns. We see that the isolines in both plots resemble hyperbolas $s \cdot k = \text{const.}$, which indicates that $m_{k,s,n,\alpha,\beta}$ scales with $s\cdot k$ rather than $s+k$ - that indicates soft recovery needs asymptotically as many measurements as exact recovery. This is also not surprising - after all, the results obtained in this section do not use any low-rank assumptions at all.

In practice however, also non-asymptotic reductions are relevant. This aspect is dealt with in Figure \ref{fig:ThresholdRatios}, where we fix $k=10$, $n=100$, $\alpha = \tfrac{\pi}{10}$ and plot the quotient $\tfrac{m_{k,s,n,\alpha,\beta}}{m_{k,s,n,0,0}}$  - for an average size and large column, respectively. Note that $m_{k,s,n,0,0}$ exactly corresponds to the threshold amount of measurements needed for exact recovery. Also note that since $m_{k,s,n,\alpha,\beta}$ only is an upper bound on the amount of measurements needed, it is not a contradiction that it is larger than $m_{k,s,n,0,0}$ for small values of $s$. We see that in particular in the case of a large column, we need considerably less measurements to ensure soft recovery than to ensure exact recovery.

Finally, in Figure \ref{fig:angleReaction}, we fix $k=s=10$ and $n=100$ and plot $m_{k,s,n,\alpha,\beta}$ depending on the size of $\alpha$ for an average size and large column. We see that at least for small $\alpha$, our result provides a smaller upper bound on the measurements needed for soft recovery than for exact recovery. Note that $m_{k,s,n,\alpha,\beta}$ is growing for large $\alpha$ (for really large $\alpha$, it even surpasses $m_{k,s,n,0,0}$. This is again not a contradiction - we have only provided an upper bound, and we already addressed that some of the estimation we make become worse as $\alpha$ (and therefore also $\beta$) grows.

% Energy concentration

\subsection{Energy Concentration}

The condition for soft recovery described above is not quite enough to secure that not only the directions, but also the magnitudes, of the columns in the minimizer $\widehat{Z}$ are close to the ones in $Z_0$, but almost. The following proposition holds.

\begin{prop}
Let $Z_0=z_0.H_0$ be supported on the set $S$ and  $\hat{z}.\widehat{H}$ be the minimizer of the program $\calP_{1,2}$ with $b= \calA(z_0.H_0)$. Assume that for each $i \in S$, $\omega(h^0_i, \hat{h}_i) \leq \alpha$, that there for some $i^*$ exists a vector $p \in \R^m$ with the properties \eqref{eq:SoftCond1}-\eqref{eq:SoftCond4}, and additionally
\begin{align} \label{eq:SoftCond5}
	\norm{A_i^*p}_2 \leq \gamma, \quad i \notin S
\end{align}
for some $\gamma <1$. Then the vector $\hat{z}$ obeys
\begin{align} \label{eq:OffSupp}
	\norm{\hat{z}_{S^c}}_1 \leq \frac{1-\cos \alpha}{\cos \alpha(1- \gamma)} \hat{z}_{i^*}
\end{align}
and
\begin{align} \label{eq:OnSupp}
	\norm{\hat{z}_S-z^0_{S}}_2 \leq \sin \alpha \frac{\max_{i \in S} \norm{A_i}}{\sigma_{\R^S}(A_{H_0})}\norm{z_0}_1 + \frac{\max_{i \in [n]} \norm{A_i} }{\sigma_{\R^S}(A_{H_0})}\norm{\hat{z}_{S^c}}_1.
\end{align}
$\sigma_{\R^S}(A_{H_0})$ denotes the to $\R^S$ restricted singular value of $A_{H_0}$, i.e.
\begin{align*}
	\sigma_{\R^S}(A_{H_0}) = \min_{ \norm{x}_2=1 , \supp x \sse S} \norm{A_{H_0}(x)}_2,
\end{align*}
and $\norm{\cdot}$ denotes the operator norm.
\end{prop}

\begin{proof}
	Let us begin by proving \eqref{eq:OffSupp}. Since $\hat{z}.\widehat{H}$ obeys the linear constraint, $z_0.H_0 - \hat{z}.\widehat{Z} \in \ker \calA = \ran \cal A^\perp$. Consequently, $\sprod{\calA^*p, z_0 .H_0 -\hat{z}.\widehat{H} }=0$, i.e.
	\begin{align*}
		\norm{z_0}_1 \leq \sum_{i \in S} \sprod{A_i^*p, z^0_i h^0_i} =  \sum_{i \in [n]} \sprod{A_i^*p, \hat{z}_i \hat{h}_i} \leq \tfrac{1}{\cos \alpha}\hat{z}_{i^*} + \sum_{i \in S\backslash \set{i^*}} \hat{z}_i + \gamma \sum_{i \notin S}\hat{z}_i = \left(\tfrac{1}{\cos \alpha}-1\right) \hat{z}_{i^*} + \norm{\hat{z}_S}_1 + \gamma \norm{\hat{z}_{S^c}}_1
	\end{align*}
	We used \eqref{eq:SoftCond1}, \eqref{eq:SoftCond4} and \eqref{eq:SoftCond5}. At the same time, due to the optimality of $\hat{z}. \widehat{H}$, there must be $\norm{\hat{z}_S}_1 + \norm{\hat{z}_{S^c}}_1 \leq \norm{z_0}_1$. These two inequalities imply 
	\begin{align*}
		\cos\alpha (1- \gamma) \norm{\hat{z}_{S^c}}_1 \leq (1-\cos \alpha)  \hat{z}_{i^*} 
	\end{align*}
	which implies the first half of  \eqref{eq:OffSupp}. 
	
	To prove the second estimate, we again utilize that $\calA(z_0.H_0 - \hat{z}. \widehat{H})=0$. That equation can namely be written
	\begin{align*}
		\sum_{i \in S}A_i^*h_i^0 (z^0_i - \hat{z}_i) + & \sum_{i \in S} A_i^*(h_i^0 -\hat{h}_i) \hat{z}_i - \sum_{i \notin S} \hat{z}_i A_i^*\hat{h}_i =0 \Rightarrow \\
		\norm{A_{H_0}(z_0 - \hat{z})}_2   \leq & \sum_{i \in S} \hat{z}_i\norm{A_i^*(h_i^0 -\hat{h}_i)}  + \sum_{i \notin S} \hat{z}_i \norm{A_i^*\hat{h}_i} \\
		\leq& \norm{\hat{z}_S}_1 \max_{i \in S} \norm{A_i} \sin(\alpha) + \norm{\hat{z}_{S^c}}_1 \max_{i \notin S} \norm{A_i} \\
		\leq & \sin(\alpha)\max_{i \in S} \norm{A_i}  \norm{z_0}_1 + (1- \sin \alpha) \max_{i \in [n]} \norm{A_i}  \norm{\hat{z}_{S^c}}_1
	\end{align*}
	We utilized that $\norm{h^0_i-\hat{h}_i}_2 \leq \sin \alpha$ for $i \in S$ and $\norm{\hat{z}_S} = \norm{\hat{z}}_1 - \norm{\hat{z}_{S^c}}_1 \leq \norm{z_0}_1 - \norm{\hat{z}_{S^c}}$. To finish the proof, we simply have to note that $\norm{A_{H_0}(z_0-\hat{z}_S)}_2 \geq \sigma_{\R^S}(A_{H_0}) \norm{(z_0 - \hat{z})_S}_2$ .
\end{proof}

\begin{rem}
	It is known that with high probability, $\max_{i \in [n]} \norm{A_i} \leqsim \sqrt{m} + \sqrt{k}$ and $\sigma_{\R^S}(A_{H_0}) \geqsim \sqrt{m} - \sqrt{s}$ \cite[Th. 9.6]{MathIntroToCS}. Hence, if we choose $m \geq C^2 \max(s,k)$, we  can guarantee that $$\frac{\max_{i \in [n]} \norm{A_i} }{\sigma_{\R^S}(A_{H_0})} \leq \frac{C+1}{C-1} $$ with high probability.
\end{rem}

%Low-Dimensional approximation properties

\subsection{Taking Low Rank Into Account.}
		
		Until know, all theory has been valid for arbitrary column-sparse matrices. In the following, we will prove that under the assumption that $Z_0$ has been softly recovered by the program $\calP_{1,2}$ and $Z_0$ is a certain type of an $r$-rank matrix (we will specify the exact requirements on $Z_0$ later), the space spanned by the leading $r$ left singular vectors of the recovered matrix will be close to the true range of $Z_0$. The argumentation will rely heavily  on a classical result from  pertubation theory; the so-called $\sin \theta$-theorem. We will start this section by recalling that theorem. We follow the original paper \cite{wedin1972}.
		
		Let two matrices $A$ and $B=A+T\in \R^{m,n}$ be given. Write the singular value decomposition of $A$ in the following manner:
		\begin{align*}
			A= A_1 +A_2 = U_1(A) \Sigma_1(A) V_1(A)^* + U_2(A) \Sigma_2(A) V_2(A)^* 
		\end{align*}
		where the left and right singular vectors of $A$ are given by the columns of the matrices $[U_1(A), U_2(A)]$ and $[V_1(A), V_2(A)]$, respectively, and the union of the diagonals of $\Sigma_1$ and $\Sigma_2$ is the set of singular values of $A$. The dimensions of the matrices are as follows:
		\begin{align*}
			U_1(A) \in \R^{m,r}, \quad & \Sigma_1(A) \in \R^{r,r}, \quad  V_1(A) \in \R^{n,r}\\
			 U_2(A) \in \R^{m,m-r}, \quad & \Sigma_2(A) \in \R^{m-r,m-r}, \quad  V_2(A) \in \R^{n,m-r}
		\end{align*}
		$B_1$, $B_2$, $U_1(B)$, $\Sigma_1(B)$ and so on are defined in the same manner. The $\sin \theta$-theorem is a statement about the principal angle between the subspaces spanned by $U_1(A)$ and $U_1(B)$ (and $V_1(A)$ and $V_1(B)$). The principal angle $\angle (E,F) \in [0, \tfrac{\pi}{2}]$ between two subspaces $E$ and $F$ are thereby defined through
		\begin{align*}
			\sin \angle(E,F) = \norm{\Pi_{E^\perp}\Pi_{F}}.
		\end{align*}

		% Motivation: converges to a solution to something which $Z_0$ very likely is (but it does not need to be the unique element of this form!)
		\begin{theo}[The $\sin \theta$-theorem] \cite{wedin1972}
			Assume the matrices $A, B$ and $T$ are given as above. Define the parameter $\mu$ through
			\begin{align*}
				\mu = \max( \norm{ \Pi_{\ran B_1} T}, \norm{T \Pi_{\ker B_1^\perp}}).
			\end{align*}
			Further assume that there exists a $\tau\geq 0$ and a $\delta >0$ such that
			\begin{align*}
				\sigma_{\min}(A_1) \geq \tau + \delta \ \wedge \ \sigma_{\max}(B_2)\leq \tau.
\end{align*}
Then there holds
\begin{align*}
	\max( \sin \angle (\ran A_1, \ran B_1),  \sin \angle (\ker A_1^\perp, \ker B_1^\perp )) \leq \frac{\mu}{\delta}.
\end{align*}			 
		\end{theo}
		
		With the above theory, we can prove the following result. We use the notation $\sprod{A}:=\ran A$.
		\begin{prop} \label{prop:softRangeStability}
			Let the matrix $Z_0 = z_0.H_0$ be supported on the set $S$ and have at most rank $r$. Assume furthermore that $S$ can be partitioned into $L$ disjoint subsets $S_\ell$, $\ell=1, \dots, L$ so that
			\begin{align*}
				h_i^0 =  \eta_\ell^0  \mod \pm 1 ,\quad i \in S_\ell
			\end{align*}
			for some unit-norm frame $(\eta_\ell)_{\ell \in [L]}$ of $\langle H_0 \rangle$ with lower  frame bound $\Lambda$.
			
			If then $\widehat{Z}=\hat{z}.\widehat{H}$ has the properties that $\omega(h_i^0, \hat{h}_i)\leq \alpha < \tfrac{\pi}{2}$, we have for the space $\langle\widehat{H}\rangle_r$ spanned by the $r$ first singular vectors of $\widehat{H}$
			\begin{align*}
				\sin \angle( \langle\widehat{H}\rangle_r, \sprod{H_0} )  \leq \frac{\max(\sin \alpha\norm{\hat{z}_S}_2,  \norm{\hat{z}_{S^c}}_2)}{\sqrt{\Lambda  \min_{\ell} \norm{\hat{z}_{S_\ell}}_2^2 - \sin{\alpha}\norm{\hat{z}_S}_2^2}}
			\end{align*}
		\end{prop}

		\begin{proof}
			Let $\epsilon >0$ be arbitrary and define the matrices $B$ and $T$ through
			\begin{align*}
				B(i) = \begin{cases} \max(\hat{z}_i,\epsilon) \sprod{\hat{h}_i,h^0_i} h^0_i & i \in S \\ 0 & i \notin S\end{cases} \quad 
				T(i) = \begin{cases} \hat{z}_i \hat{h}_i - \max(\hat{z}_i,\epsilon) \sprod{\hat{h}_i,h^0_i} h^0_i & i \in S \\ \hat{z_i} \hat{h}_i & i \notin S \end{cases}
			\end{align*}
			Adopting the notation from above, we further set $				A= \hat{z}. \widehat{H}$.
			Then $A = B+T$ and, continuing to use the notation from above,  $\sprod{A_1} = \langle \widehat{H}\rangle_r$, $\sprod{B_1} = \sprod{H_0}$. (The latter is due to $\max(\hat{z}_i,\epsilon) \langle{\hat{h}_i,h^0_i}\rangle\neq 0$ for all $i$, which is a consequence of $\alpha < \tfrac{\pi}{2}$ and $\epsilon >0$.) Hence
			\begin{align*}
				\sin\angle( \sprod{H_0},\langle \widehat{H} \rangle_r,  ) = \sin\angle(\sprod{A_1}, \sprod{B_1}).
\end{align*}		
	 We now want to apply the $\sin \theta$-theorem. Towards this end, we have to estimate the parameters $\mu$, $\tau$ and $\delta$. $\tau$ can be chosen equal to zero -- the range of $B$ is contained in $\sprod{H_0}$, and hence is at most $r$-dimensional. Consequently, $B$ has at most $r$ non-zero singular values, and $B_2=0$. To estimate $\mu$, we begin by noting that $\ker B_1^\perp = \ker B^\perp \sse \R^S$ (since $\R^{S^c} \sse \ker B$). This implies
	 \begin{align*}
	 	\norm{T \Pi_{\ker B_1^\perp}} \leq \norm{T \Pi_{\R^S}} \leq \sup_{\norm{x}_2 \leq 1} \sum_{i \in S} \abs{x_i}( \epsilon\abs{ \langle \hat{h}_i, h_i^0\rangle} + \hat{z}_i \norm{\Pi_{\sprod{H_0}^\perp}\hat{h}_i}_2 ) \leq \sqrt{s}\epsilon +  \norm{\hat{z}_S}_2 \sin \alpha. 
	 \end{align*}
		In the last step, we used the Cauchy Schwarz inequality and \begin{align*}
			\norm{\Pi_{\sprod{H_0}^\perp}\hat{h}_i}_2 = \min_{v \in \sprod{H_0}} \norm{\hat{h}_i -v }_2 \leq \norm{\hat{h}_i- \langle{\hat{h}_i,h_i^0}\rangle h_i^0}_2 = \sqrt{1 - \cos^2(\omega(\hat{h}_i,h_i^0))} \leq \sin \alpha
		\end{align*} due to $\omega(\hat{h}_i, h_i^0)\leq \alpha$ for $i \in S$. We also have, since $\sprod{B_1} = \sprod{H_0}	$
		\begin{align*}
			\norm{\Pi_{\ran B_1} T} = \norm{\hat{z}_{S^c}.(\Pi_{\sprod{H_0}}\widehat{H})} \leq \norm{\hat{z}_{S^c}}_2.
		\end{align*}
		Here we used that the vectors in $\widehat{H}$ have unit norm and an argument similar as above. Summarizing, $$\mu \leq \max ( \sqrt{s}\epsilon + \sin \alpha\norm{\hat{z}_S}_2 ,\norm{\hat{z}_{S^c}}_2)$$
		
		It remains to estimate the gap $\delta$. Since we already have noted that $\tau$ can be chosen equal to zero we mearly need to estimate $\sigma_{\min}(A_1)= \sigma_r(A)$ from below. To do this, we use two well known facts: firstly the equality $\sigma_r(A) = \sigma_r(A^*)$, and secondly the so-called $\max$-$\min$-principle (or \emph{Courant-Fischer Theorem}):
		\begin{align*}
			\sigma_r(A^*) = \max_{ \dim V \leq r } \min_{\substack{p \in V \\ \norm{p}_2 = 1} } \norm{A^*p}.
		\end{align*}
		Hence, since $\dim \sprod{H_0} \leq r$
		\begin{align*}
			\sigma_r(A^*) \geq \min_{\substack{p \in \sprod{H_0} \\ \norm{p}_2 =1}} \norm{(\hat{z}. \widehat{H})^*p}_2.
		\end{align*}
	The adjoint of $\hat{z}.\widehat{H}$ is given by $((\hat{z}.\widehat{H})p)_i= \hat{z}_i \langle{\hat{h}_i, p}\rangle$. Hence
	\begin{align*}
		\norm{(\hat{z}. \widehat{H})^*p}_2^2 =\sum_{i \in [n]} \hat{z}_i^2 \abs{\langle{\hat{h}_i,p}\rangle}^2 \geq \sum_{i \in S} \hat{z}_i^2 \abs{\langle{\hat{h}_i,p}\rangle}^2  
	\end{align*}
	Now we notice that if $\norm{p}_2=1$, the following is true due to the triangle inequality of $\omega$ and $\omega(\hat{h}_i, h_i^0)\leq \alpha$
	\begin{align*}
		\abs{\langle{\hat{h}_i, p}\rangle}^2 = \cos^2(\omega(\hat{h}_i, p)) \geq \cos^2(\min(\omega(h_i^0,p)+\alpha, \tfrac{\pi}{2})) \geq \cos^2(\omega(h_i^0,p)) - \sin\alpha = \abs{\langle h^0_i, p\rangle }^2 - \sin \alpha,
	\end{align*}
	where the proof of the last inequality is postponed to the Appendix (more specifically Lemma \ref{lem:cosSinIneq}). Therefore, we may estimate
	\begin{align*}
	 \sum_{i \in S} \hat{z}_i^2 \abs{\langle{\hat{h}_i,p}\rangle}^2  \geq  \sum_{i \in S} \hat{z}_i^2 \abs{\langle{h_i^0,p}\rangle}^2   - \sin(\alpha) \norm{\hat{z}_S}_2^2
	\end{align*}
	 Now we utilize the structure of the set $(\hat{h}_i)_{ i \in S}$ to estimate
	\begin{align*}
		\sum_{i \in S} \hat{z}_i^2 \abs{\langle{h_i^0,p}\rangle}^2 = \sum_{\ell \in [L]}   \abs{\langle{\eta_\ell^0,p}\rangle}^2 \sum_{i \in S_\ell} \hat{z}_i^2 \geq \min_{\ell} \norm{z_{S_\ell}}_2^2 \sum_{\ell \in [L]}   \abs{\langle{\eta_\ell^0,p}\rangle}^2 \geq \Lambda  \min_{\ell} \norm{z_{S_\ell}}_2^2
	\end{align*}
for all $p \in \sprod{H_0}$ with $\norm{p}_2$. All in all,
\begin{align*}
	\delta^2 \geq \Lambda  \min_{\ell} \norm{z_{S_\ell}}_2^2 - \sin(\alpha)\norm{\hat{z}_S}_2^2,
\end{align*}
which together with the $\sin \theta$-theorem proves
\begin{align*}
\sin \angle( \langle\widehat{H}\rangle_r, \sprod{H_0} )  \leq \frac{\max( \epsilon \sqrt{s} + \sin \alpha\norm{\hat{z}_S}_2,  \norm{\hat{z}_{S^c}}_2)}{\sqrt{\Lambda  \min_{\ell} \norm{\hat{z}_{S_\ell}}_2^2 - \sin{\alpha}\norm{\hat{z}_S}_2^2}}.
\end{align*}
Letting $\epsilon \to 0$ yields the claim.
		\end{proof}
		
		\begin{rem}
			In the case that $r=1$, all columns $h_i^0$, $i \in S$ must be  equal modulo sign, which implies that we may choose $L=1$, $\eta_1^0=h_1^0$, $S_1=S$ and consequently $\Lambda=1$. We then arrive at the cleaner estimate
			\begin{align*}
				\sin \angle( \langle\widehat{H}\rangle_r, \sprod{H_0} )  \leq \frac{\max(\sin \alpha\norm{\hat{z}_S}_2,  \norm{\hat{z}_{S^c}}_2)}{ \norm{\hat{z}_S}_2\sqrt{1-\sin \alpha}}.
			\end{align*}
		\end{rem}

% Two ideas

\section{Heuristic Proposals How One Could Utilize Soft Recovery} \label{sec:heur}

% Thresholding

In this section, we will present two ideas how one could use the phenomenon of soft recovery to  design a recovery algorithm for matrices which are both column-sparse and have low rank. These ideas are of highly heuristic nature, whence we mostly study them experimentally.

\subsection{Thresholding}

The above analysis indicates that the set of indices where $\hat{z}_i$ is large probably coincides relatively well with the support of $z_0$, since $\norm{\hat{z}_{S^c}}_1$ is small. Hence, if we choose a set $\widehat{S}$ for which $\norm{\widehat{Z}_{\widehat{S}}}_{1,2} \geq \norm{\widehat{Z}}_{1,2}*\tau$ for $\tau \approx 1$ and $\widehat{S}\geq s$, $\hat{S}$ is most probably a small set which still contains the support of the ground truth signal $Z_0$. After having identified the support, the remaining task is then to recover a low-rank matrix of dimension $k \times \abs{\widehat{S}}$, which we choose to do with nuclear norm-minimization minimization. This will be successful with high probability already when $m \geqsim r(k + \abs{\widehat{S}})\log(k \abs{\widehat{S}})$ \cite{recht2010guaranteed}.

The procedure outlined above is summarized in Algorithm \ref{alg:12ThresNuc}.

\begin{algorithm}
	\caption{(NAST) Nuclear norm After Soft recovery Thresholding.} \label{alg:12ThresNuc}
		\KwData{ A linear map $\calA: \R^{k,n} \to \R^q$, a sparsity $s$, a parameter $\tau \in (0,1)$ and a vector $b$.}
		\KwResult{ An estimate $Z_*$ of a sparse solution of $\calA(Z)=b$, where $Z_*=z_*.H_*$ with $z_*$ sparse and $H_*$ low rank.}
		
		\nl Solve the $\calP_{1,2}$ for a column-sparse matrix $\widehat{Z}= \hat{z}.\widehat{H}$. 
		
		\nl $\widehat{S} \leftarrow$ smallest set $S$ with at least $s$ indices so that $\norm{\widehat{Z}_S}_{1,2} \geq \tau \norm{\widehat{Z}}_{1,2}$.

		\nl Solve the nuclear minimization problem
		\begin{align*}
			\min_{ \supp Z \sse \widehat{S}} \norm{Z}_{*} \text{ subject to } \calA(Z) =b
		\end{align*}
		and output the solution $Z_*$.
\end{algorithm}

\subsubsection{Numerical Experiments}

To test the performance of the $NAST$-algorithm, we perform a small numerical experiment. For $m=100,105, \dots, 300$, we generate random $10$-column sparse matrices in $\R^{10,100}$ of the form $\sum_{i=1}^r h_i z_i^*$ for $r=1, 2$ and $5$ (in all experiments $s=k=10$, $n=100$.) The vectors $h_i$ are independent and uniformly distributed on the sphere where as $z$ has a uniformly drawn random support and normally distributed and independent non-zero entries. We measure these matrices with a randomly drawn Gaussian measurement matrix, and test if $NAST$ is able to recover the ground truth signal. The minimization problems were solved with help of the $MATLAB$-package {\tt cvx} \cite{cvx}. A success is declared if the relative error in Frobenius norm is less than $0.1 \%$. This experiment was repeated a hundered times, and the results are depicted in Figure \ref{fig:NAST}. We can clearly see that the $NAST$-procedure outperforms $\ell_{1,2}$. We also see that the performance is better for low ranks, as expected. We performed a small control experiment for $m=300$ testing the performance of nuclear norm minimization - out of a hundered trials, there was not a single success.

\begin{figure}
	\centering 
	\includegraphics[scale=.4]{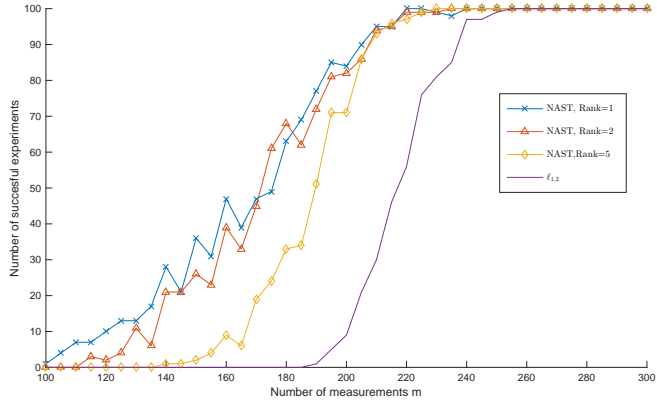}
	\caption{Results of numerical performance tests of $NAST$. \label{fig:NAST}}
\end{figure}

% '' Column streamlining ''

\subsection{''Column Streamlining''} \label{sec:ColStream}
Another idea, which is a bit more involved, we want to present is the following: If $Z_0$ is of rank $r$, the columns containing most of the energy of $\widehat{Z}$ all almost lie in an $r$-dimensional subspace, not far from $\ran Z_0$. Hence, the best rank $r$-approximation $V$ of $\ran \widehat{Z}$ should be close to $\ran Z_0$.

This information can be used to modify the $\ell_{1,2}$-norm to more greatly penalize components of the norms of $Z$ not lying in the space $V$. We propose the following way of doing this
\begin{align*}
	\min \norm{Z} + \norm{\Pi_{V^\perp} Z}_{1,2} \text{ subject to } \calA(Z)=b.
\end{align*}
The procedure of alternately calculating matrices $Z$ and subspaces $V$ according to the procedure described above is then repeated until some stopping criterion is met. Some heuristic proposals for this criterion are for instance that the difference between iterate $Z_k$ and $Z_{k+1}$ drops below  some threshold $\epsilon$, or the same for the $(r+1)$:st singular value of $Z_{k}$ (indicating that we have found a low-rank solution), or also the $(s+1)$:st largest column (indicating that we have found a sparse solution). Alternatively, one could break the iteration already when, say, $2s$ columns are larger than some threshold, form a set $\widehat{S}$ of the corresponding indices, and solve a low-rank problem with that constraint, as above. The main algorithm is summarized in Algorithm \ref{alg:ColumnStreamLining}. We have chose to call it the \emph{Column Streamlining} algorithm, since it, when successful, forces the non-zero columns of the iterates to all align in a common, low-dimensional subspace.

\begin{algorithm}
		\caption{Column Streamlining} \label{alg:ColumnStreamLining}
		\KwData{ A linear map $\calA: \R^{k,n} \to \R^m$, a rank $r$ and a vector $b$.}
		\KwResult{ An estimate $Y_*$ of a sparse solution of $\calA(Y)=b$, where $Y_*=z.H$ with $z$ sparse and $H$ of rank less than $r$.}
		Initalize $Y_0$ as the solution of $\calP_{1,2}$ for $\calA$ and $b$, and $V$ as the the best $r$-dimensional approximation of $\ran Y_0$, i.e $$V=\spn(U_1, \dots U_r),$$ where $U\Sigma V^*$ is the $SVD$ of $Y_0$.
		
		 \Repeat{A stopping criterion is met}{
		\nl  $Y_q$ $\leftarrow$ $\argmin \norm{Y}_{1,2} + \Vert{\Pi_{V_{q-1}^\perp}Y}\Vert_{1,2}$ subject to $\calA(Y)=b$.
		 
		\nl $V_q$  $\leftarrow$ The best $r$-dimensional approximation of $\ran Y_k$.
		}
		\end{algorithm}
		
		We cannot report much success about theoretical guarantees about the performance of the Column Streamlining algorithm. We however believe that it again will not perform well using the optimal amount $r(s+k)$ measurement, but instead again get stuck at the $(s\cdot k)$--bottleneck. The argument goes as follows: Let us assume that the iterates $Y_q \to Z_0$. Then it is not hard to prove that $Z_0$ has to solve the minimization problem
		\begin{align}
			\argmin \norm{Z}_{1,2} + \Vert{\Pi_{\sprod{H_0}^\perp}Z}\Vert_{1,2} \text{ subject to } \calA(Z)=b. \tag{$\calP_{1,2,\sprod{H_0}}$}
		\end{align}
		
			\begin{figure}
\centering
\includegraphics[scale=.3]{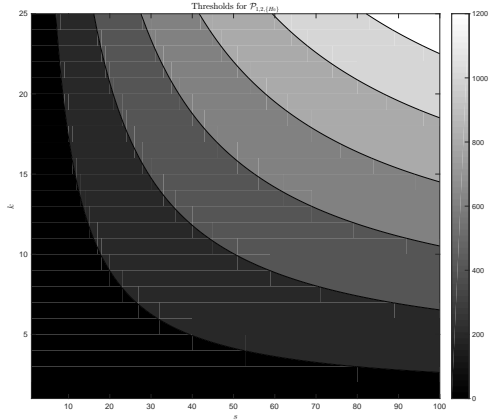}
\caption{Plot of the function \eqref{eq:ColStreamThreshold} for $n=100$ and $r=1$.\label{fig:colStreamThresholds}}
\end{figure}
	
By following the exact same route as in Section \ref{sec:statDim}  one may now easily prove that an upper bound of the amount of Gaussian measurements needed for the success of the program $\calP_{1,2,\sprod{H_0}}$ is given by
\begin{align}
	\mu_{k,s,n,r} = \inf_{\tau >0 } s(r+\tau^2 + \erw{ \pos(\norm{g^{k-r}}_2 - \tau)^2}) + (n-s) \erw{\pos(\norm{g^{k}}_2 - \tau)^2}. \label{eq:ColStreamThreshold}
\end{align}
(The only hard part of the argument, which is to calculate the subdifferential of $\norm{\cdot}_{1,2} + \Vert{\Pi_{\sprod{H_0}^\perp} \cdot}\Vert_{1,2}$,  is presented in Section \ref{sec:ColStreamApp} of the Appendix.) By plotting this function in Figure  \ref{fig:colStreamThresholds} for $n=100$, $r=1$ for varying values of $s$ and $k$, we see that it scales as $s\cdot k$ rather than as $s+k$. It should however be noted that in absolute numbers, we need considerably less measurements to secure success of $\calP_{1,2,\sprod{H_0}}$ than for $\calP_{1,2}$ - as an example, $\mu_{10,10,100,1}\approx 130$ whereas $m_{10,10,100,0}\approx 236$

\subsubsection{Numerical Experiments}

We chose to test the numerical performance of the Column Streamlining algorithm as follows: Matrices were generated as above (i.e. $k=s=10$, $n=100$) for $r=2$, for each $m=100, 101, \dots, 140$. Then we let the Column Streamlining algorithm perform $10$ iterations, and subsequently recorded the relative difference in Frobenius norm between the final iterate $Y_*$ and the ground truth signal $Z_0$. This was repeated a hundered times for each value of $m$. We again used {\tt cvx} to solve the minimization programs. The number of experiments in which the final relative error was smaller than $0.1 \%$ and $1 \%$, respectively, are depicted and compared with the results of the $NAST$-experiments in Figure \ref{fig:NASTvsColStream}. We see that the Column Streamlining outperforms the $NAST$-approach. The big difference between the number of final iterates with a relative error smaller than $0.1\%$ and $1 \%$ suggest that better results if we let the algorithm perform more iterations. 

\begin{figure}
\centering
\includegraphics[scale=.3]{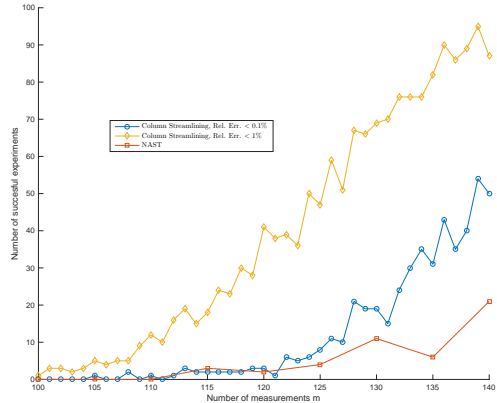}
\caption{\label{fig:NASTvsColStream} Numerical results of the Column Streamlining Algorithm.}
\end{figure}

\subsection*{Acknowledgement}
The author acknowledges support from  Deutsche
Forschungsgemeinschaft (DFG) Grant KU 1446/18~-~1 and Berlin Mathematical School. He also likes to thank Gitta Kutyniok, Saeid Haghighatshoar and Peter Jung for interesting discussions on this and other topics.

\bibliographystyle{abbrv}
\bibliography{bibliographyCSandFriends}

 \appendix
\section{Appendix}

We here provide minor proofs not contained in the main body of the article. We begin with the argument advertised in the introduction about the impossibility of soft recovery via $\ell_1$-minimization.
\begin{prop} \label{prop:l1supp} 
	Let $A \in \R^{m,n}$ and $z_0 \in \R^n$ be arbitrary. If $\ell_1$-minimization problem $\calP_1$ with $b= Az_0$ has a unique solution $z_*$, one of the following alternatives hold
	\begin{itemize}
		\item $z_0=z_*$
		\item $z_*(i) = 0$ for at least one $i \in \supp z_0$ . 
	\end{itemize}
\end{prop}
\begin{proof}
%	Denote $S=\supp z_0$ and suppose that neither of the two alternatives is true, i.e. that there exists a solution $z_*$ of $\calP_1$ with $\sgn(z_*(i))=\sgn(z_0(i))$ for all $i \in S$. We distinguish two cases
% Strong duality then implies that there exists a $p \in \R^m$ with $\norm{z_*}_1=\sprod{p,b} = \sprod{A^*p, z^*}$ and $\norm{A^*p}_\infty \leq 1$, i.e. if we denote the $i$:th column of $A$ with $a_i$
%	\begin{align*}
%		\sprod{a_i,p} &= \sgn (z^*(i)), i \notin \supp z_* \\
%		\abs{\sprod{a_i,p}} &\leq \sgn (z^*(i)), i \notin \supp z_* 	.						
%	\end{align*}
%	But since $S \sse \supp z_*$ and $z_0$ and $z_*$ have the same sign pattern on $S$, the existence of $p$ readily implies that $z_0$ solves $\calP_1$, which is a contradiction.

Denote $\supp z^* = S^*$. Towards a contradiction, assume that $z_0 \neq z_*$ and that the support $S$ of $z_0$ is contained in $S_*$. By \cite[Lem. 5.1]{Flinth2015PROMP} (which little more than a specification of \cite[Thm. 1]{DonohoTanner2005}), $\tfrac{z_*}{\norm{z_*}_1}$ then necessarily lies in the relative interior of an $(\abs{S^*}-1)$-face $F$ of the crosspolytope $$C_1 = \set{ x \in \R^n \ \vert \ \norm{x}_1\leq 1}$$ which is mapped to an $(\abs{S^*}-1)$-face of $AC_1$. In particular, the columns $(a_i)$ of $A$ corresponding to $i \in S^*$ are linearly independent. But due to the linear constraint, we have
\begin{align*}
	b= \sum_{i \in S} z_0(i) a_i = \sum_{i \in S^*} z_*(i) a_i.
\end{align*}
This equation, together with $S \sse S_*$ and the linear independence,   readily implies $z_0 =z_*$.
\end{proof}

Next, we calculate the dual problem of $\calP_1^+$, as was needed in the proof of Lemma \ref{lem:AHopt}. Note that this calculation is relatively standard, but we include it for completeness.

\begin{lem}\label{lem:DualPosL1}
The dual problem of
\begin{align*}
	\min \norm{z}_1 = \sum_{i \in [n]} z_i \text{ subject to } Ax=b \ , \ x\geq 0
\end{align*}
is equivalent to
\begin{align*}
	\max \sprod{b,p} \text{ subject to } A^*p \leq \ind,
\end{align*}
where $\ind$ is the vector consisting only of entries equal to $1$
\end{lem}
\begin{proof}
	Let us begin by writing the Lagrange dual
	\begin{align*}
		\calL(z, p, \mu) =  \norm{z}_1+ \sprod{p, b-Az} - \sprod{\mu, z} = \sprod{b,p} + \norm{z}_1 - \sprod{z, A^*z + \mu}
	\end{align*}
	We see that $\inf_z \calL(z,p,\mu)$ is finite exactly when $\norm{A^*p + \mu }_{\infty} \leq 1$, in which case it is equal to $\sprod{b,p}$. The dual problem is hence
	\begin{align*}
		\max_{p, \mu \geq 0} \sprod{b,p} \text{ subject to } \norm{A^*p + \mu}_\infty \leq 1.
	\end{align*}
	Due to the fact that there exists a $\mu \geq 0$ with $\norm{A^*p + \mu} \leq 1$ if and only if $A^*p \leq \ind$, we see the stated equivalence.
\end{proof}

Finally, we prove the trigonometric inequality used in the proof of Proposition \ref{prop:softRangeStability}.

\begin{lem} \label{lem:cosSinIneq} For $x \in [0, \pi]$ and $\alpha \in [0, \tfrac{\pi}{2}]$, we have
\begin{align*}
	\cos^2\left( \min\left(x+\alpha, \tfrac{\pi}{2}\right)\right) \geq \cos^2(x) - \sin \alpha
\end{align*}
\end{lem}
\begin{proof}
	We treat the cases $x+\alpha \leq \tfrac{\pi}{2}$ and $x+\alpha \geq \tfrac{\pi}{2}$ separately. In the first case, we have
	\begin{align*}
		\cos^2\left( \min\left(x+\alpha, \tfrac{\pi}{2}\right)\right) &= \cos^2(x+\alpha) = (\cos(x)\cos(\alpha) - \sin(x) \sin( \alpha))^2 \\
		&= \cos^2(x) \cos^2(\alpha) - 2 \cos(x) \sin(x) \cos(\alpha) \sin(\alpha) + \sin^2(x) \sin^2(\alpha) \\
		&= \cos^2(x) - \sin(\alpha) \left( (\cos^2(x)-\sin^2(x))\sin(\alpha) + 2\cos(x)\sin(x)\cos(\alpha) \right) \\
		&= \cos^2(x) -\sin(\alpha)\left( \cos(2x)\sin(\alpha) + \sin(2x)\cos(\alpha) \right) \\
		&= \cos^2(x) - \sin(\alpha) \sin(2x+\alpha) \geq \cos^2(x) - \sin(\alpha).
	\end{align*}
In the second case, we may argue as follows: $\alpha$ and $\tfrac{\pi}{2}$ both lie in the interval $[-\tfrac{\pi}{2},\tfrac{\pi}{2}]$, in which $\sin$ is increasing. The inequality $\alpha \geq \tfrac{\pi}{2}-x$ therefore implies $\sin(\alpha) \geq \sin\left( \tfrac{\pi}{2} -x \right) = \cos(x)$. Hence
\begin{align*}
	\cos^2(x) -\sin(\alpha) \leq \cos^2(x)-\cos(x) \leq 0 = \cos^2\left( \min\left(x+\alpha, \tfrac{\pi}{2}\right)\right),
\end{align*}
which proves the claim.
\end{proof}

\subsection{The Subdifferential of the Norm Used In $\calP_{1,2,\sprod{H_0}}$} \label{sec:ColStreamApp}
Here we provide the calculation of the subdifferential of $\norm{\cdot}_{1,2} + \Vert{\Pi_{\sprod{H_0}^\perp} \cdot}\Vert_{1,2}$ at $Z_0$, as was advertised in \ref{sec:ColStream}. 
	
		\begin{prop} Let $Z_0=z_0.H_0$ be supported on the set $S$. Then the subdifferential of $\norm{\cdot}_{1,2} + \norm{\Pi_{\sprod{H_0}^\perp}\cdot}_{1,2}$  at $Z_0$ is given by the set of matrices $V$ for which
		\begin{align*}
			V(i) &\in h_i^0 +  B_1(0) \cap \sprod{H_0}^\perp, \quad i \in S\\
			V(i) &\in B_1(0) +  B_1(0) \cap \sprod{H_0}^\perp, \quad i \notin S .
		\end{align*}
		\end{prop}
\begin{proof}
	We are looking for the set of matrices $V$ for which
	\begin{align*}
		\sum_{i \in S} \norm{h_i^0 + U(i)}_2 + \norm{\Pi_{\sprod{H_0}^\perp}U(i)}_2 + \sum_{i \notin S} \norm{U(i)}_2 + \norm{\Pi_{\sprod{H_0}^\perp}U(i)}_2 \geq \sum_{i \in S} \norm{h_i^0}_2 + \sprod{V(i),U(i)} + \sum_{i \notin S} \sprod{V(i), U(i)} 
	\end{align*}
	for all $U \in \R^{k,n}$. We used that $h_i^0 \in \sprod{H_0}$ for all $i \in S$. Since the columns of $U$ can be chosen independently, we can treat each index separately. In the following, we drop the indices in the calculations for notational simplification, and use lower case letters to denote columns of $V$ and $U$, respectively.
	
	{\bf Case 1: $i \in S$.} We aim to characterize the vectors $v \in \R^k$ for which
	\begin{align} \label{eq:subdiffPiV}
		 \norm{h^0 + u}_2 + \norm{\Pi_{\sprod{H_0}^\perp}u}_2 \geq \norm{h^0}_2 + \sprod{v,u}
	\end{align}
	for all vectors $u \in \R^k$. If $v \in  h^0 +  B_1(0) \cap \sprod{H_0}^\perp$, \eqref{eq:subdiffPiV} is satisfied:
	\begin{align*}
	\norm{h^0}_2 + \sprod{v,u}= \norm{h^0}_2 + \sprod{h^0,\Pi_{\sprod{H_0}}u} + \sprod{\Pi_{\sprod{H_0}^\perp}v,\Pi_{\sprod{H_0}^\perp}u} & \leq \norm{h^0 + \Pi_{\sprod{h^0}}u}_2 + \norm{\Pi_{\sprod{H_0}^\perp}u}_2 \\
	&\leq \norm{h^0 + u}_2 + \norm{\Pi_{\sprod{H_0}^\perp}u}_2.
	\end{align*}
	We now only need to argue that \eqref{eq:subdiffPiV} implies that $v \in h^0 +  B_1(0) \cap \sprod{H_0}^\perp$. By plugging in $\pm h^0$ for $u$, we see that
	\begin{align*}
		\sprod{h^0,v} = \norm{h^0}_2.
	\end{align*}
	Hence, $v$ is of the form $h^0 + v_1 + v_2$ for $v_1 \in \sprod{H_0}$, $v_1 \perp h^0$ and $v_2 \perp \sprod{H_0}$. Now, by plugging in $\lambda v_1$ with $\lambda >0$ for $u$, we obtain
	\begin{align*}
		 \norm{h^0}_2 + \lambda \norm{v_1}_2^2 \leq \norm{h^0 + \lambda v_1}_2 = \sqrt{\norm{h^0}_2^2 + \lambda^2 \norm{v_1}_2^2} \Rightarrow \norm{v_1}^2 \leq \frac{\sqrt{\norm{h^0}_2^2 + \lambda^2 \norm{v_1}_2^2} - \norm{h^0}_2}{\lambda}, \lambda >0
	\end{align*}
	By letting $\lambda \to 0$, we obtain $\norm{v_1}_2=0$. Similarly, plugging in $\lambda v_2 $ for $u$, we obtain
	\begin{align*}
		\norm{h^0}+ \lambda \norm{v_2}_2^2 \leq \norm{h^0 + \lambda v_2}_2 + \lambda \norm{v_2}_2 \Rightarrow (\norm{v_2}_2^2- \norm{v_2}_2) \leq \frac{\sqrt{\norm{h^0}_2^2 + \lambda^2 \norm{v_2}_2^2} - \norm{h^0}_2}{\lambda}, \lambda >0.
	\end{align*}
Again by letting $\lambda \to 0$, we obtain $(\norm{v_2}_2^2- \norm{v_2}_2) \leq 0$, which only is satisfied if $\norm{v_2}_2\leq 1$.

	{\bf Case 2: $i \notin S$.} Here the aim is to characterize the vectors $v$ for which
	\begin{align} \label{eq:subdiffPiVOff}
		 \norm{u}_2 + \norm{\Pi_{\sprod{H_0}^\perp}u}_2 \geq \sprod{v,u}
	\end{align}
	is true for every $u \in \R^k$. It is not hard to see that \eqref{eq:subdiffPiVOff} is satisfied for every $v$ of the form $w_1 + w_2$ with $\norm{w_1}_2 \leq 1$ and $\norm{w_2}_2 \leq 1$, $w_2 \perp \sprod{H_0}$. 
	
	To see that \eqref{eq:subdiffPiVOff} implies that $v$ can be written as claimed is a bit more tricky. It is clear that we can always write $v$ in the form $\mu \eta_1 + \lambda \eta_2$ with $\eta_1 \in \sprod{H_0}$ and $\eta_2 \perp \sprod{H_0}$ both having unit norm. Hence, $v = (\mu \eta_1 + (\lambda -1) \eta_2) + \eta_2$ and the proof is finished as soon as we have argued that $\mu^2 + (\lambda-1)^2 \leq 1$. However, plugging in $\mu \eta_1 + (\lambda -1) \eta_2$ for $u$ in \eqref{eq:subdiffPiVOff} yields
	\begin{align*}
		\sqrt{\mu^2 + (\lambda -1)^2} + (\lambda -1) \geq \mu^2 + \lambda(\lambda -1) \Rightarrow \sqrt{\mu^2 + (\lambda -1)^2} \geq \mu^2 + (\lambda -1)^2,
	\end{align*}
	which implies $\mu^2 + (\lambda-1)^2 \leq 1$.
\end{proof}
\end{document}